\documentclass[twoside,11pt]{article}
\usepackage{mlmod}
\ShortHeadings{Risk Management Perspective on Statistical Estimation}{Javeed, Kouri and Surowiec}
\firstpageno{1}

\ifpdf
\hypersetup{
  pdftitle={Risk Management Perspective on Statistical Estimation},
  pdfauthor={D. P. Kouri and T. M. Surowiec}
}
\fi

\RequirePackage{enumitem}
\setlist[enumerate]{leftmargin=.5in}
\setlist[itemize]{leftmargin=.5in}
\RequirePackage{amsmath,amssymb,amsopn,bbm}
\RequirePackage[OT1]{fontenc}
\RequirePackage{appendix,natbib}
\RequirePackage{tikz,tikz-cd,pgfplots,color}
\usetikzlibrary[positioning,arrows,decorations.pathmorphing,backgrounds,fit,calc,matrix,shapes]

\newcommand{\risk}{\mathcal{R}}

\newcommand{\cB}{\mathcal{B}}

\newcommand{\cL}{\mathcal{L}}
\newcommand{\cM}{\mathcal{M}}

\newcommand{\cR}{\mathcal{R}}

\newcommand{\cU}{\mathcal{U}}
\newcommand{\cV}{\mathcal{V}}

\newcommand{\fA}{\mathfrak{A}}
\newcommand{\fD}{\mathfrak{D}}
\newcommand{\fP}{\mathfrak{P}}
\newcommand{\real}{\mathbb{R}}
\newcommand{\bbe}{\mathbb{E}}
\newcommand{\dset}{\cL}
\newcommand{\mset}{\cM}
\newcommand{\esssup}{\operatorname*{ess\,sup}}
\newcommand{\essinf}{\operatorname*{ess\,inf}}
\definecolor{azure}{rgb}{0.0, 0.5, 1.0}

\newcommand{\eqdef}{\triangleq}

\begin{document}
\thispagestyle{empty}
\title{A Risk Management Perspective on Statistical Estimation and Generalized Variational Inference
\thanks{DPK's research was sponsored in part by Sandia National Laboratories LDRD
``Risk-Adaptive Experimental Design for High-Consequence Systems''. }}

\author{\name Aurya S.\ Javeed \email asjavee@sandia.gov \\
        \addr Optimization and Uncertainty Quantification \\
              Sandia National Laboratories \\
              P.O.\ Box 5800, MS-1320 \\
              Albuquerque, NM 87125, USA
        \AND
        \name Drew P.\ Kouri \email dpkouri@sandia.gov \\
        \addr Optimization and Uncertainty Quantification \\
              Sandia National Laboratories \\
              P.O.\ Box 5800, MS-1320 \\
              Albuquerque, NM 87125, USA
        \AND
        \name Thomas M.\ Surowiec \email thomasms@simula.no \\
        \addr Department of Numerical Analysis and Scientific Computing \\
                 Simula Research Laboratory\\
              	Kristian Augusts gate 23\\
                 0164, Oslo, Norway}


\maketitle

\begin{abstract}
Generalized variational inference (GVI) provides an optimization-theoretic
framework for statistical estimation that encapsulates many traditional
estimation procedures. The typical GVI problem is to compute a distribution of
parameters that maximizes the expected payoff minus the divergence of the
distribution from a specified prior.  In this way, GVI enables likelihood-free
estimation with the ability to control the influence of the prior by tuning
the so-called learning rate.  Recently, GVI was shown to outperform traditional
Bayesian inference when the model and prior distribution are misspecified.  In
this paper, we introduce and analyze a new GVI formulation based on utility
theory and risk management.  Our formulation is to maximize the expected payoff
while enforcing constraints on the maximizing distribution. We recover the
original GVI distribution by choosing the feasible set to include a constraint
on the divergence of the distribution from the prior.  In doing so, we
automatically determine the learning rate as the Lagrange multiplier for the
constraint. In this setting, we are able to transform the infinite-dimensional
estimation problem into a two-dimensional convex program. This reformulation
further provides an analytic expression for the optimal density of parameters.
In addition, we prove asymptotic consistency results for empirical
approximations of our optimal distributions.  Throughout, we draw connections
between our estimation procedure and risk management.  In fact, we demonstrate
that our estimation procedure is equivalent to evaluating a risk measure.  We
test our procedure on an estimation problem with a misspecified model and prior
distribution, and conclude with some extensions of our approach. 
\end{abstract}




 \noindent\fbox{%
     \parbox{0.97\textwidth}{
         \begin{center}
         \medskip
             \emph{\bfseries This is a ``working paper.'' A shorter, streamlined version with a different title and updated references and numerical experiments will appear in a future edition.}
         \medskip
         \end{center}
     }
 }

\section{Introduction}
Many statistical estimation problems can be formulated as the optimization problem
\begin{equation}\label{eq:intro-0}
  \max_{\theta\in\Theta}
    \,\left\{M_N(\theta)\eqdef\frac{1}{N}\sum_{n=1}^N m(\theta,y_n)\right\},
\end{equation}
where $\{y_n\}_{n=1}^N\subset\Upsilon$ are realizations of the noisy data $Y$,
$\Upsilon$ is the data space, $\theta\in\Theta$ are the parameters to be
estimated, and $m:\Theta\times\Upsilon\to\real$ is a pay-off or log-likelihood,
see \citep{SAvanDeGeer_2000a,AWvanDerVaart_2000a} and the
references therein. This class of problems encapsulates M-estimation and
empirical risk minimization problems, including regression and maximum
likelihood estimation as well as other training applications in machine learning. 
On the other hand, empirical approximations of traditional stochastic
optimization problems also have the form \eqref{eq:intro-0}, see e.g.,
\citep{AShapiro_DDentcheva_ARuszczynski_2014a}.

In this paper, we study an extension of \eqref{eq:intro-0} that permits
measure-valued solutions.  This coherent generalized variational inference
(CGVI) problem is given by
\begin{equation}\label{eq:MLE}
  \max_{P\in\fA}\,
    \bbe_P\left[M_N\right],
\end{equation}
where $\fA$ is a set of admissible probability measures defined on subsets of
$\Theta$ and $\bbe_P$ denotes the expectation with respect to the probability
measure $P$.  Classically, the CGVI problem \eqref{eq:MLE} models the situation
in which $\theta$ are random {\em incidental} parameters
\citep{kiefer1956consistency}.  In contrast, we view \eqref{eq:MLE} as a
distributionally robust
version of \eqref{eq:intro-0} that enables the inclusion of prior
statistical information regarding the unknown deterministic {\em structural}
parameters $\theta$ through the definition of the set $\fA$. The idea for
distributionally robust optimization goes back to \citep{JZackova_1966}, but has
seen increased interest in recent years, see, e.g.,
\citep{DPKouri_2017a,DKuhn_etal_2019a,AShapiro_2017a}.

One practical consequence of \eqref{eq:MLE} is improved computational
tractability.  Many optimization algorithms
only guarantee convergence to locally optimal solutions, while global methods
often scale poorly with problem size, leading to intractability for large-scale
nonconcave problems. Consequently, if $M_n$ is not concave in $\theta$, then it
may be computationally infeasible to find a global maximum of
\eqref{eq:intro-0}. On the other hand, \eqref{eq:MLE} is a concave maximization
problem as long as $\fA$ is convex.  We may therefore view \eqref{eq:MLE}
as a concave relaxation of \eqref{eq:intro-0}.  Owing to the concavity of
\eqref{eq:MLE}, locally optimal solutions to \eqref{eq:MLE} are also
globally optimal and hence, local optimization methods can be used.
Unfortunately, \eqref{eq:MLE} is an infinite-dimensional optimization problem,
which may be impossible to solve analytically or numerically without
discretization. However, as is commonly done in distributionally robust
optimization \citep{DKuhn_etal_2019a,AShapiro_2017a}, we can choose $\fA$ in
such a way that \eqref{eq:MLE} can be equivalently reformulated as a low
dimensional convex minimization problem.  To achieve this, we will employ
concepts from decision theory, including utility functions and risk measures,
to quantify our maximum tolerance for deviating from a prescribed prior
probability measure $\Pi$ on $\Theta$.

A common method for comparing two probability measures is to use a
$\phi$-divergence \citep{ICsiszar_1963a,TMorimoto_1963a}.  The class of 
$\phi$-divergences includes the popular Kullback-Liebler (KL), $\chi^2$, Hellinger,
and R\'{e}nyi divergences.  As demonstrated in
\citep{ABenTal_MTeboulle_1986a,ABenTal_MTeboulle_1987,ABenTal_MTeboulle_2007a},
there is a fundamental link between $\phi$-divergences, utility functions, and
risk measures, a relationship that we will build on. In particular, we will
show that for a given utility functional $\cU$ acting on a space of random
variables $\mset$, we can define the disutility functional $\cV(X)=-\cU(-X)$
for $X \in \cM$ and ultimately a statistical divergence as $\Phi(p)=\cV^*(p)$,
where $p$ is the density of a probability measure $P$ with respect to the prior
$\Pi$ and $\cV^*$ denotes the Fenchel conjugate of $\cV$.  This construction is
closely related to risk measures via optimized certainty equivalents
\citep{ABenTal_MTeboulle_2007a} and the risk quadrangle
\citep{RTRockafellar_SUryasev_2002a}.
%
This link between utility functions and statistical divergences is practically
appealing as utility functions have been intensely studied for decades, e.g.,
\citep{KJArrow_1970,PCFishburn_1988,JWPratt_1964,JvNeumann_OMorgenstren_2007},
and are perhaps more intuitive to choose than a statistical divergence $\Phi$.

Some of the most striking results associated with our decision-theoretic
framework for estimation arise when the utility functional $\cU$ is an expected
utility, i.e., $\cU(X) = \bbe[u(X)]$, where $u : \real\to\real$ is a
scalar utility function.  Many well-known utility functions $u$ give rise to
popular $\phi$-divergences (cf.\ \citep{ABenTal_MTeboulle_1987}). For example,
the ubiquitous exponential utility function generates the KL
divergence, while the isoelastic utility function generates the R\'{e}nyi
divergence.  In this setting, we define $\fA$ as a subset of probability
measures whose $\phi$-divergence from the prior measure $\Pi$ is less than a
prescribed tolerance $\varepsilon>0$.  With this $\fA$, \eqref{eq:MLE} is closely
related to the Gibbs posterior
\citep{PGBissiri_CCHolmes_SGWalker_2016,WJiang_MATanner_2008,TZhang_2006a,TZhang_2006b},
and more generally generalized variational inference (GVI) as introduced in
\citep{JKnoblauch_JJewson_TDamoulas_2019}.
Unlike traditional Bayesian inference, our approach allows for a misspecified
prior $\Pi$, is not predicated on a specific noise structure (model
misspecification), and is computationally tractable once prior samples are
available.  In fact, these three features were the primary motivation for
GVI and the so-called Rule of Three \citep{JKnoblauch_JJewson_TDamoulas_2019}.
There is other related work on measure-valued M-estimators, e.g.,
\citep{ABasu_BGLindsay_1994,ABasu_HShioya_CPark_2011,LGreco_WRacugno_LVentura_2008}.
However, \citep{LGreco_WRacugno_LVentura_2008} considers a Bayesian posterior
based on M-estimation using pseudo-likelihoods, which is less general than our
approach, and \citep{ABasu_BGLindsay_1994,ABasu_HShioya_CPark_2011} develop an
estimation theory in which divergences are also employed, but the
measure-valued M-estimators assume a parametric form.  Similarly,
$\phi$-divergences and Wasserstein distances have recently found broad
applicability in machine learning 
\citep{arjovsky17a,AGretton_KMBorgwardt_MJRasch_BSchoelkopf_ASmola_2012,DKuhn_etal_2019a,NIPS2016_6066};
particularly for constructing generative adversarial
networks (GANs) to estimate probability laws. 

Using the utility-based statistical divergence $\Phi$ described above, we
express our tolerance for deviation from $\Pi$ by restricting $\fA$ to a
smaller class of $\Pi$-absolutely continuous densities $p$ that satisfy
$\Phi(p)\le\varepsilon$ and denote this set by $\fD_\varepsilon$.
By replacing $\fA$ with $\fD_\varepsilon$, we show that \eqref{eq:MLE} can
be reformulated as the two-dimensional convex stochastic program
\begin{equation}\label{eq:MLE-0}
  \min_{\lambda\ge 0,\,\mu\in\real}\; \left\{\mu + \varepsilon\lambda + (\lambda\Phi)^*(M_N-\mu)\right\},
\end{equation}
where $(\lambda\Phi)^*$ denotes the Fenchel conjugate of $\lambda \Phi$.  In fact,
we prove an explicit form for the CGVI density that depends on the
optimal $\mu$ and $\lambda$, and is related to convex subgradients
of $(\lambda\Phi)^*$.
In many practical cases, the CGVI density has a simple and computable form,
reducing the infinite-dimensional optimization problem \eqref{eq:MLE}
to a two-dimensional convex problem, regardless of the concavity and
smoothness of $M_N$ and the dimension of $\Theta$.

As mentioned above, the $\phi$-divergence approach to CGVI with $\phi$ giving rise to the
KL divergence is closely related to the Gibbs posterior.
The Gibbs posterior was originally used to describe statistical
mechanics phenomena
\citep{RLDobrushin_1968a,RLDobrushin_1968b,HOGeorgii_2011,OELandfordIII_DRuelle_1969},
but has recently been applied in the setting of information theory and
estimation \citep{TZhang_2006a,TZhang_2006b}
and Bayesian statistics \citep{PGBissiri_CCHolmes_SGWalker_2016,WJiang_MATanner_2008}.  
The Gibbs posterior estimation problem associated with \eqref{eq:intro-0} is
\begin{equation}\label{eq:gibbs}
  \max_{p\in\fD}\;\sigma^{-1}\bbe_\Pi\left[p M_N\right] - D_{\rm KL}(p \| \pi)
\end{equation}
where $\sigma>0$ is the learning rate, $\mathfrak{D}$ is
the set of probability density functions on $\Theta$ with respect to the
prior measure $\Pi$, which is assumed to have density $\pi$, and
$D_{\rm KL}(\cdot\|\cdot)$ is the KL divergence.  This
problem was originally proposed and analyzed in the pioneering works
\citep{ICsiszar_1975,MDDonsker_SRSVaradhan_1975,AZellner_1988}, which
considered variational representations of Bayesian inference. The optimal
solution to \eqref{eq:gibbs} is given in closed-form by
\[
  \bar{p}(\theta) = \frac{\exp\left(\sigma^{-1} M_N(\theta)\right)}
     {\bbe_\Pi\left[\exp\left(\sigma^{-1} M_N\right)\right]}.
\]
As seen in this expression, when the pay-off function $M_N$ is a log-likelihood
function, the Gibbs posterior $\bar{p}$ recovers the traditional Bayesian
posterior by setting $\sigma=1$ and multiplying with the prior density $\pi$.
Note that \eqref{eq:MLE} with the constraint $D_{\rm KL}(p\|\Pi)\le\varepsilon$
is a \textit{constrained} optimization problem and \eqref{eq:gibbs} is an
\textit{unconstrained} problem.  In fact, \eqref{eq:gibbs} can be thought
of as an approximation to \eqref{eq:MLE} where the $D_{\rm KL}(p \| \pi)$ term
penalizes for the inequality constraint.

The learning rate $\sigma$ calibrates the influence of the likelihood or payoff
function on the posterior. A significant amount of effort has recently gone
into determining the ``best'' learning rate for a given problem. We refer the
reader to
\citep{PGruenwald_2012,gruenwald2018inconsistency,CCHolmes_SGWalker_2017,JWMiller_DBDunson_2019,NSyring_RMartin_2019,wu2020comparison},
the latter of which contains numerous related references and provides a
numerical study comparing the various perspectives.   We will later observe
that the parameter $\lambda$ in \eqref{eq:MLE-0}, which is the Lagrange
multiplier for the constraint $\Phi(p) \le \varepsilon$, is an automatic
choice for $\sigma$ when $\Phi$ is the KL divergence.

Though the Gibbs posterior is well-studied, it was shown in
\citep{JKnoblauch_JJewson_TDamoulas_2019,knoblauch2019frequentist} that it may
not perform well under a misspecified prior, which is often considered the
rule and not the exception in machine learning.  In fact, the papers
\citep{JKnoblauch_JJewson_TDamoulas_2019,knoblauch2019frequentist} show that
the R\'{e}nyi divergence can outperform the KL divergence when the prior is
misspecified.  This raises the question as to whether other
$\phi$-divergences have similar properties or special uses. Though we do
not limit ourselves purely to $\phi$-divergences, these observations provide
further justification for our study.

The remaining sections are organized as follows. In Section~\ref{sect:GGE}, we
introduce the basic notation and assumptions as well as some tools from convex
and functional analysis.  Section~\ref{sect:genphi} is dedicated to risk-averse
estimation.  We begin by highlighting the connections between
$\phi$-divergences and classical utility theory. Using the general theory
developed in the first subsection, we then prove the existence of CGVI densities
and the link between GVI and CGVI. We then demonstrate the existence of and
provide a general formula for CGVI densities in the case of the
$\phi$-divergence.  Afterwards, we discuss the connections to risk measures
and provide four explicit examples of $\phi$-divergence CGVI densities.
Following these theoretical results, we analyze the asymptotic consistency for
empirical approximations of CGVIs, including the effects of the large data limit
and sampling limit on the computable densities in Section~\ref{sec:comp}. We
conclude this section with a numerical example that demonstrates the ability of
CGVIs to outperform both traditional Bayesian posteriors and the Gibbs
posterior when the model and prior is misspecified. Finally, we conclude in
Section~\ref{sec:empir} by introducing an extension to ``Empirical CGVI'' for
situations in which we are unsure which utility function to choose, but for
which a large set of observations $Y$ is available and perhaps a few
observations of $\theta$.

\section{Background}

In this section, we provide the required background information to analyze
GVI problems in the context of risk management.  We first provide the
notation and problem assumptions used throughout this paper and then review
the concepts of $\phi$-divergence, regret and risk functionals.

\subsection{Notation and Problem Assumptions}\label{sect:GGE}

We denote by $\cB\subseteq 2^\Theta$ a $\sigma$-algebra on $\Theta$ and by
$\fP=\fP(\Theta,\cB)$ the set of probability measures on measurable space
$(\Theta,\cB)$. We assume that we have a
probability measure describing our prior knowledge of $\theta$, denoted by
$\Pi\in\fP$ and that the probability space $(\Theta,\cB,\Pi)$ is complete.
For notational convenience, we denote the expectation with respect to $\Pi$
simply by $\bbe$.  For arbitrary $P\in\fP$, we denote the expectation with
respect to $P$ by $\bbe_P$. In the subsequent discussion, we require the set of
probability density functions with respect to the prior measure $\Pi$, which
we denote by
\[
  \fD \eqdef \{\, p:\Theta\to[0,\infty]\,\vert\,p\;\text{is $\cB$-measurable and }
            \bbe[p]=1\} \subset L^1(\Theta,\cB,\Pi).
\]
In particular, for every $P\in\fP$ that is absolutely continuous with respect
to $\Pi$ (denoted $P\ll \Pi$), there exists $p\in\fD$ such that
$\mathrm{d}P = p\, \mathrm{d}\Pi$, i.e., $p$ is the Radon-Nikodym derivative of
$P$ with respect to $\Pi$. After analyzing the GVI problem through the lense of
risk management, we introduce a related GVI problem, which we call coherent
GVI or CGVI for short, of the form
\begin{equation}\label{eq:MLE0}
  \max_{P\in\fA}\,
    \bbe_P\left[M_N\right],
\end{equation}
where $\fA\subseteq\fP$ is an appropriately selected subset of probability
measures. 
Clearly, if $\fA=\{\Pi\}$, then the maximizing measure is $\Pi$
(i.e., full confidence in the prior).  On the other hand, if $\fA=\fP$, then
the maximizing measures are convex combinations of point masses centered at the
maximizers to \eqref{eq:intro-0} (i.e., no confidence in the prior)
\citep[Th.~7.37]{AShapiro_DDentcheva_ARuszczynski_2014a}.  The key to CGVI
is in determining reasonable and application-relevant classes of probability
measures $\fA$, which can be done using utility theory and risk management.
One common and inuitive form for $\fA$ is
\begin{equation}\label{eq:div}
  \fA = \{\,P\in\fP\,\vert\,D(P\| \Pi)
    \le\varepsilon\,\},\quad\varepsilon \ge 0,
\end{equation}
where $D:\fP\times\fP\to[0,\infty]$ is a statistical metric or divergence
\citep{STRachev_1991}.  In words, \eqref{eq:MLE0} seeks a maximizer of the
total expected pay-off over the set of probabilities measures that are within
$\varepsilon$ of the prior measure $\Pi$.  Note that if $\varepsilon = 0$, then
$\fA=\{\Pi\}$ and if $\varepsilon=\infty$, then $\fA=\fP$.

In the subsequent section, we require the following techincal definitions.
We denote by $\dset$ and $\mset$ two Banach spaces of $\cB$-measurable
functions from $\Theta$ into $\real$ for which
\[
  \bbe[|p X|] < \infty \quad\forall\, p\in\dset,\;X\in\mset.
\]
We assume that $\dset$ and $\mset$ contain all simple functions on $\Theta$
and we associate with $\mset$ and $\dset$ the bilinear form
$\langle p,X\rangle\eqdef\bbe[p X]$ for $p\in\dset$ and $X\in\mset$, making
$(\mset,\dset,\langle\cdot,\cdot\rangle)$ paired topological vector spaces.  We
denote the topological dual space of $\dset$ by $\dset^*$ and similarly for
$\mset$. We abuse notation and denote the canonical embedding of $\dset$ into
$\mset^*$ simply by $\dset\subseteq \mset^*$ and the canonical embedding of
$\mset$ into $\dset^*$ by $\mset\subseteq\dset^*$.  We denote the weak topology
on $\mset$ induced by $\dset$ by $\sigma(\mset,\dset)$ and the weak topology on
$\dset$ induced by $\mset$ by $\sigma(\dset,\mset)$.  We will restrict this
abstract setting to two situations: $\dset=\mset^*$ and $\mset=\dset^*$, the
latter presenting additional, significant, theoretical complications.  The
first case represents a majority of use cases such as
$\mset=L^s(\Theta,\cB,\Pi)$ with $s\in[1,+\infty)$ or $\mset$ is an
appropriately defined Orlicz space.  The second case covers the situation when
$\mset=L^\infty(\Theta,\cB,\Pi)$ is paired with $\dset=L^1(\Theta,\cB,\Pi)$
endowed with the weak$^*$ topology.

Throughout, we require a handful of common concepts from convex and variational
analysis.  Let $\mathcal{X}$ be a Banach space with topological dual space
$\mathcal{X}^*$ and duality pairing
$\langle\cdot,\cdot\rangle_{\mathcal{X}^*,\mathcal{X}}$.  The Fenchel conjugate
of the function $f:\mathcal{X}\to[-\infty,+\infty]$ is defined by
\[
  f^*(\eta) \eqdef \sup_{x\in\mathcal{X}} \{\langle \eta, x\rangle_{\mathcal{X}^*,\mathcal{X}} - f(x)\}
   \quad\forall\,\eta\in\mathcal{X}^*
\]
and we employ the convention that the Fenchel conjugate of a function
$g:\mathcal{X}^*\to[-\infty,+\infty]$ is defined on $\mathcal{X}$ rather than
on all of $\mathcal{X}^{**}$.  The function $f$ is said to be proper if
$f(x)>-\infty$ for all $x\in\mathcal{X}$ and there exists $x_0\in\mathcal{X}$
such that $f(x_0)<\infty$.  We denote the effective domain of $f$ by
$\mathrm{dom}\,f \eqdef \{x\in\mathcal{X}\,\vert\,f(x) < \infty\}$.
The function $f$ is said to be closed with respect to a topology on
$\mathcal{X}$ (e.g., weak or strong) if its epigraph,
$\mathrm{epi}\,f \eqdef \{(x,t)\in\mathcal{X}\times\real\,\vert\, f(x) \le t\}$,
is closed in the respective product topology on $\mathcal{X}\times\real$.
Recall that if $f$ is convex, then the subdifferential of $f$ at
$x\in\mathcal{X}$ is defined by
\[
  \partial f(x) \eqdef \{\eta\in\mathcal{X}^*\,\vert\, f(x') - f(x) \ge \langle \eta, x'-x\rangle_{\mathcal{X}^*,\mathcal{X}}\;\;\forall\,x'\in\mathcal{X}\}.
\]
If $f$ is proper, closed and convex, then $f = f^{**}$, where $f^{**}$ denotes
the Fenchel conjugate of $f^*$.  Moreover, the Fenchel-Young inequality states
that
\begin{equation}\label{eq:fy-ineq}
  f(x) + f^*(\eta) \ge \langle \eta, x \rangle_{\mathcal{X}^*,\mathcal{X}}
\end{equation}
and equality holds in \eqref{eq:fy-ineq} if and only if $\eta\in\partial f(x)$,
or equivalently if $x\in\partial f^*(\eta)$.  Finally, for a set
$C\subset\mathcal{X}$, we denote by $\delta_C:\mathcal{X}\to[-\infty,+\infty]$
the indicator function of $C$ and by
$\sigma_C:\mathcal{X}^*\to[-\infty,+\infty]$ the support function of $C$. That
is,
\[
  \delta_C(x)\eqdef\left\{\begin{array}{ll}
    0 & \text{if $x\in C$} \\
    +\infty & \text{otherwise}
  \end{array}\right.
  \qquad\text{and}\qquad
  \sigma_C(\eta) \eqdef \sup_{x\in C}\; \langle \eta, x\rangle_{\mathcal{X}^*,\mathcal{X}}.
\]

\subsection{Connections Between the $\phi$-Divergence and Risk Management}

In this subsection, we review disutility (or regret) functions and their
close relation to the $\phi$-divergence \citep{ICsiszar_1963a,TMorimoto_1963a}.
We say that a function $v:\real\to[-\infty,+\infty]$ is a
{\em risk-averse disutility function} if it is convex, increasing, and
satisfies $v(x)\ge x$ for all $x\in\real$ and $v(x)=x$ implies $x=0$.  Two
popular risk-averse disutility functions are the exponential disutility
function
\begin{equation}\label{eq:exputil}
  v(x) = \frac{(\exp(r x)-1)}{r}, \quad r > 0,
\end{equation}
and the isoelastic disutility function
\begin{equation}\label{eq:isoutil}
  v(x) = \left(1+\frac{x}{\beta}\right)^{\beta}-1, \quad \beta < 0.
\end{equation}
Since $v$ is increasing, convex, and satisfies $v(x) \ge x$ for all $x\in\real$
and $v(0)=0$, we have that its Fenchel conjugate $v^*$ is convex, $v^*(1) = 0$
and $v^*(t) = +\infty$ for all $t < 0$.
These properties of $v^*$ are exactly those required of $\phi$-divergence
functions.  In particular, given a function $\phi:\real\to[0,+\infty]$ that is
convex and satisfies $\phi(1)=0$ and $\phi(t) = +\infty$ for all $t<0$, we can
define the $\phi$-divergence of a probability measure $P\in\fP$ from the prior
$\Pi$ by
\[
  D_\phi(P\|\Pi) \eqdef \left\{\begin{array}{ll}
     \bbe[\phi(p)] & \text{if $P\ll \Pi$ and $p\,\mathrm{d}\Pi = \mathrm{d}P$} \\
     +\infty & \text{otherwise}
  \end{array}\right. .
\]
One common challenge when using a $\phi$-divergence for GVI is in choosing
$\phi$ \citep{JKnoblauch_JJewson_TDamoulas_2019}---a task that is often
difficult to conceptualize.  Given the close ties between $\phi$ and a
disutility function $v$, we suggest first choosing $v$ and then defining $\phi$
as $\phi=v^*$.  Often, the concept of disutility is more intuitive and can be
tailored to the underlying application.

A popular approach in economic utility theory is to choose disutility functions
that exhibit certain curvature properties such as constant absolute risk
aversion (CARA) in the sense of Arrow and Pratt
\citep{KJArrow_1970,JWPratt_1964}.  To achieve CARA, one first defines the
absolute risk-aversion coefficient
\[
  A(x) = \frac{v''(x)}{v'(x)}
\]
and then generates a CARA disutility function by solving the ordinary
differential equation (ODE) $A(x)=c$ for some constant $c\in\real$.  A related
concept is constant relative risk aversion (CRRA), where the relative
risk-aversion coefficient is given by
\[
  R(x) = xA(x).
\]
Similarly, one can generate a CRRA disutility function by solving the ODE
$R(x)=c$ for some constant $c\in\real$.  Recall that the exponential disutility
function satisfies CARA while the isoelastic disutility function satisfies
CRRA.
See Figure~\ref{fig:utility-phi} for common
disutility functions and their associated $\phi$-divergence functions. We note
that the exponential disutility function generates the KL divergence, the
isoelastic disutility generates the R\'{e}nyi divergence, and the quadratic
disutility function generates the $\chi^2$ divergence. The fourth disutility
function in Figure~\ref{fig:utility-phi} is piecewise linear and is not risk
averse.  However, it is of interest since it generates the total variation
divergence. 

\begin{figure}[!ht]
  \centering
  \includegraphics[width=0.4\textwidth]{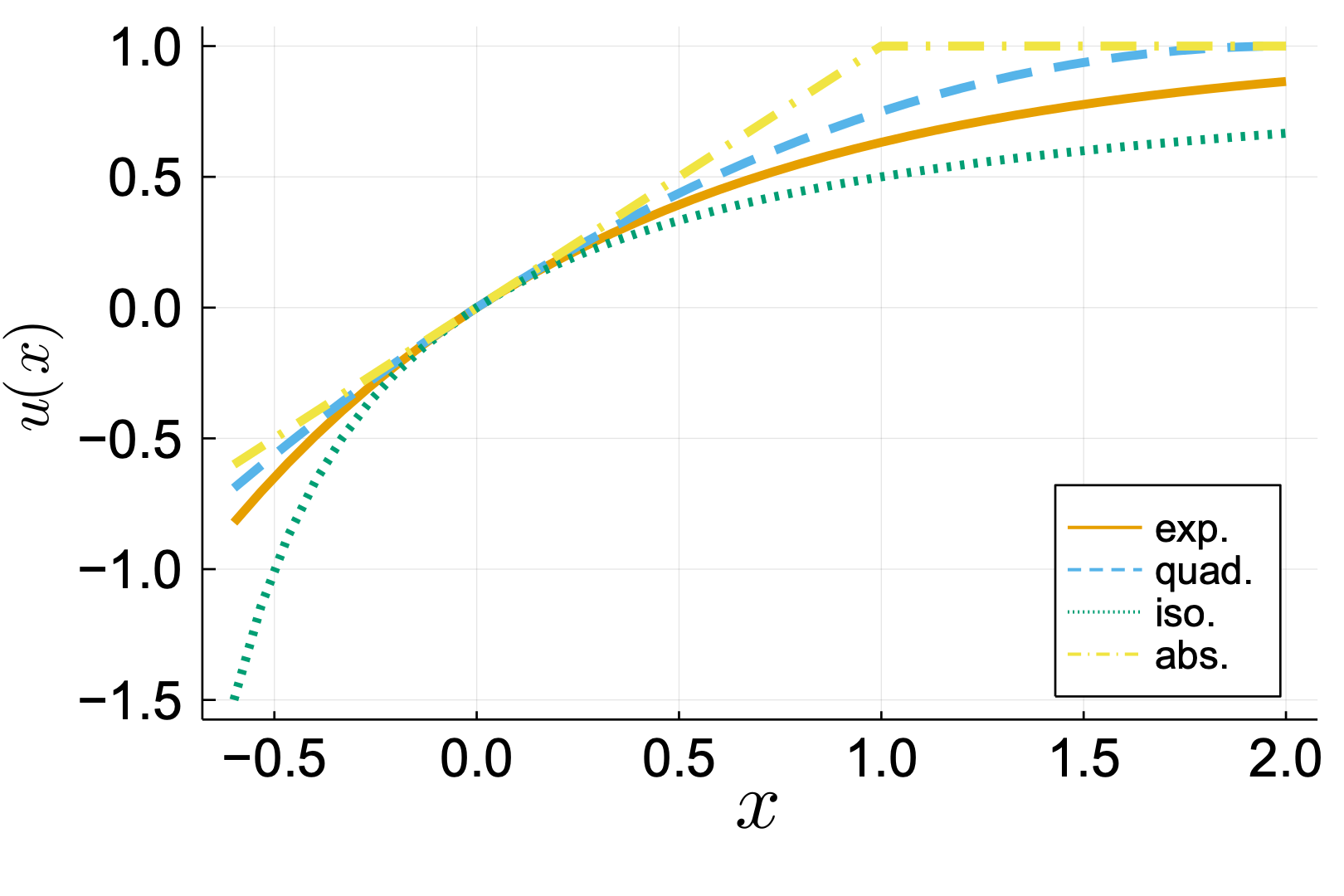}
   \includegraphics[width=0.4\textwidth]{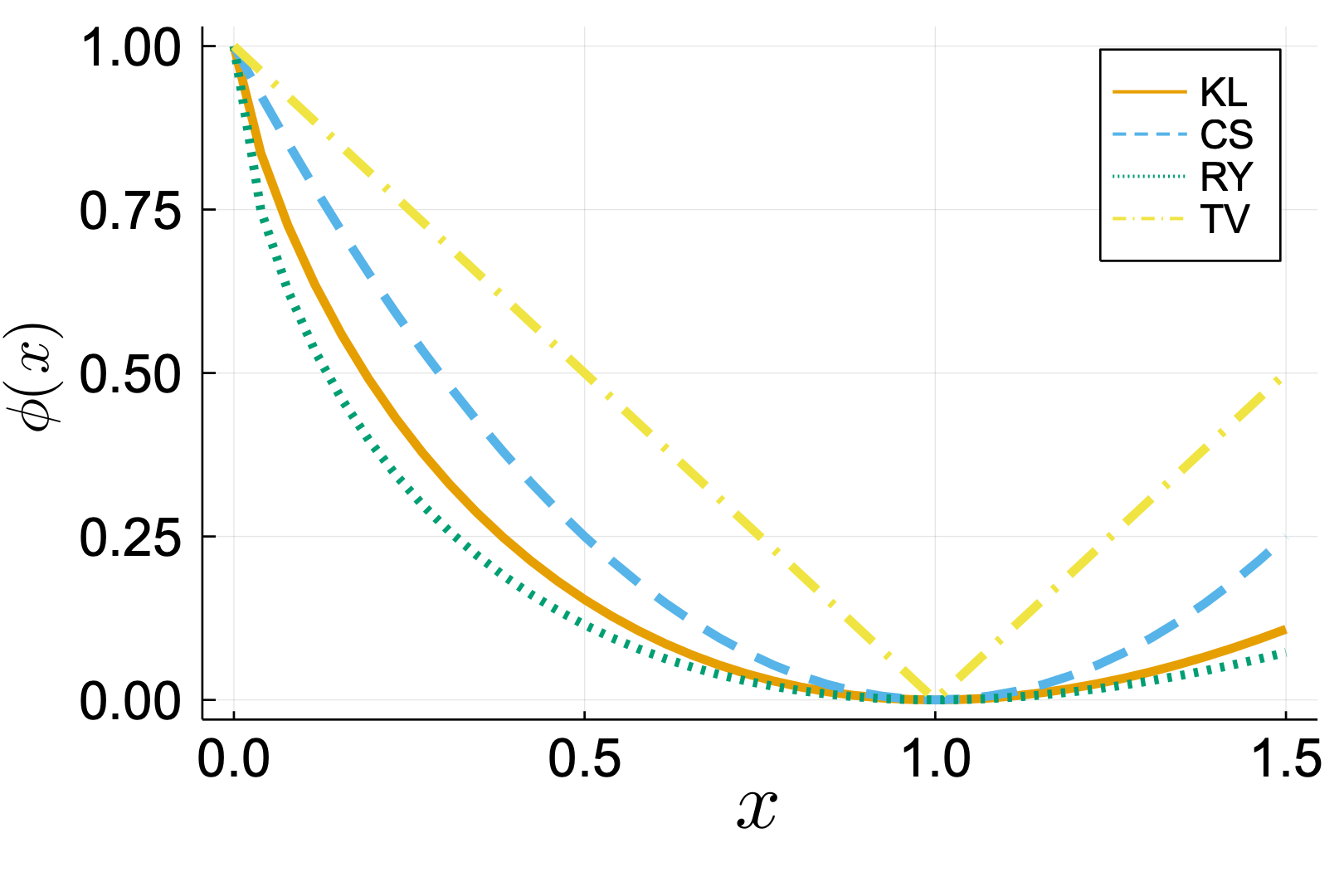}
  \caption{
Four common disutility functions and their associated
$\phi$-divergences: a) the exponential disutility (exp.) and Kullback-Leibler
divergence (KL), b) a truncated quadratic disutility (quad.) and $\chi^2$
divergence (CS), c) the isoelastic disutility and R\'{e}nyi divergence (RY), d)
an absolute disutility and TV divergence (TV). }
  \label{fig:utility-phi}
\end{figure} 

For the purpose of our theoretical results, and to facilitate the relationship
between GVI and risk measures, we work in a slightly more general setting than
traditional utility theory.  We say that the functional
$\cV:\mset\to[-\infty,+\infty]$ is a
{\em risk-averse expected disutility function} if it satisfies:
\begin{itemize}
  \item[(V1)] {\bf Convexity:} $\cV$ is proper, $\sigma(\mset,\dset)$-lower semicontinuous and convex;
  \item[(V2)] {\bf Risk Aversion:} For any $X\in\mset$, $\cV(X)\ge \bbe[X]$ and $\cV(X)=\bbe[X]$ implies $X\equiv 0$;
  \item[(V3)] {\bf Monotonicity:} If $X,\,X'\in\mset$ satisfy $X\le X'$ a.s., then $\cV(X)\ge\cV(X')$.
\end{itemize}
Commonly, $\cV$ has the form
\[
  \cV(X) = \bbe[v(X)],
\]
where $v:\real\to[-\infty,+\infty]$ is a risk-averse disutility function.
According to \cite{JvNeumann_OMorgenstren_2007}, any ``rational'' decision
maker has an expected disutility function that characterizes their decision
when faced with risky outcomes.  In their theory, the decision maker's
preference (denoted by the binary relationships $\prec$, $\succ$ and $\sim$)
follows four axioms when deciding between any three random outcomes $X$, $Y$,
and $Z$:
\begin{itemize}
  \item {\bf Completeness:} Either $X\prec Y$, $X\succ Y$ or $X\sim Y$;
  \item {\bf Transitivity:} If $X\prec Y$ ($X\sim Y$) and $Y\prec Z$ ($Y\sim Z$), then $X\preceq Z$ ($X\sim Z$);
  \item {\bf Continuity:} If $X\prec Y \prec Z$, then there exists
        $t\in[0,1]$ such that $tX + (1-t)Z \sim M$;
  \item {\bf Independence:} If $X\prec Y$, then $tX + (1-t)Z\preceq tY + (1-t)Z$ for all $t\in(0,1]$.
\end{itemize}
Although the von Neumann-Morgenstern theory ensures the existence of an
expected disutility function, it does not provide a constructive means to
determine the disutility function.  However, there have been attempts to
systematically select a decision maker's disutility function
\citep{GJWarren_2019}.

Analogously, we can generalize the $\phi$-divergence using the expected
disutility functional $\cV$ as $\Phi(p)=\cV^*(p)$. Our next result lists the
properties of $\Phi$ and demonstrates that $\Phi$ in fact defines a statistical
divergence.

\begin{proposition}\label{prop:Phi-conj}
  Let the regret functional $\cV:\mset\to[-\infty,+\infty]$ satisfy (V1)--(V3).
  Then, the functional $\Phi:\dset\to[-\infty,+\infty]$ defined by
  \[
    \Phi(p) \eqdef \cV^*(p)
  \]
  is proper, $\sigma(\dset,\mset)$-lower semicontinuous, convex and satisfies
  the following properties:
  \begin{enumerate}
    \item $\Phi(p)\ge 0$ for all $p\in\dset$;
    \item $\Phi(p)=0$ if $p\equiv 1$ a.s.;
    \item $\Phi(p)<\infty$ implies $p\ge 0$ a.s.;
    \item If $\partial\cV(0)=\{1\}$, then $\Phi(p)=0$ implies $p\equiv 1$ a.s.;
    \item If $\partial\cV(0)=\{1\}$, then
          $\mathrm{dom}\,\Phi=\{1\}$
          if and only if $\cV$ is positively homogeneous.
  \end{enumerate}
\end{proposition}
\begin{proof}
  $\Phi$ is proper, $\sigma(\dset,\mset)$-lower semicontinuous and convex by
  the definition of the Fenchel conjugate \citep[pg.~188]{RTRockafellar_1971a}.
  Property~1 follows from the Fenchel-Young inequality and the fact that
  $\cV(0) = 0$, i.e.,
  \begin{equation}\label{eq:fy-proof1}
    \Phi(p) = \Phi(p) + \cV(0) = \Phi(p) + \Phi^*(0) \ge \bbe[0 p] = 0 \quad\forall\, p\in\dset.
  \end{equation}
  Property~2 follows from (U2).  That is, $X\mapsto\bbe[X]-\cV(X)$ has a
  unique maximizer at $X\equiv 0$ with maximum value 0, and hence
  $\Phi(1)=\cV^*(1)=0$.
  Property~3 follows from (U2) and (U3).  First notice that
  \[
    \cV(X) \le \cV(0) = 0 \quad\forall\,X\le 0\;\;\text{a.s.}
  \]
  and let $p\in\dset$ be such that
  $\Theta_p=\{\theta\in\Theta\,\vert\,p(\theta)<0\}$
  has positive probability with respect to the prior, i.e., $\Pi(\Theta_p) > 0$.
  By setting
  \[
    X_p(\theta) = \left\{\begin{array}{ll}
     -1 & \text{if $\theta\in\Theta_p$} \\
     0 & \text{otherwise}
    \end{array}\right. ,
  \]
  we see that $X_p\in\mset$ since $\mset$ includes all simple functions and
  \[
    \bbe[tX_p p] - \cV(tX_p) \ge t\bbe[X_p p] > 0 \quad\forall\,t>0.
  \]
  Taking the supremum in $t>0$ demonstrates that $\Phi(p)=+\infty$, proving
  property~3.
  For property~4, the Fenchel-Young inequality \eqref{eq:fy-proof1} holds with
  equality if and only if
  \[
    0 \in \partial\Phi(p) \quad\iff\quad p \in \partial\cV(0).
  \]
  and the desired result follows since $\partial\cV(0)=\{1\}$ by assumption.
  Finally, to prove property~5, we first note that if
  $\mathrm{dom}\,\Phi=\{1\}$, then
  $\Phi$ is the indicator function of $\{1\}$ and hence
  $\cV(X)=\Phi^*(X)=\bbe[X]$, which is positively homogeneous.  On the other
  hand, if $\cV$ is positively homogeneous, then $\Phi$ is an indicator
  function.  Since $\Phi(p) > 0$ unless $p\equiv 1$ a.s., we must have that
  $\Phi(p)=\infty$ unless $p\equiv 1$ a.s.
\end{proof}

Proposition~\ref{prop:Phi-conj} demonstrates that any risk-averse expected
disutility function $\cV$ that satisfies $\partial\cV(0)=\{1\}$, defines the
statistical divergence
\[
  D_\Phi(P\|\Pi) \eqdef \left\{\begin{array}{ll}
     \Phi(p)=\cV^*(p) & \text{if $P\ll \Pi$ and $p\,\mathrm{d}\Pi = \mathrm{d}P$ with $p\in\dset$} \\
     +\infty & \text{otherwise}
  \end{array}\right. .
\]
$D_\Phi$ clearly generalizes the $\phi$-divergence since we can define
$\Phi(p)=\bbe[\phi(p)]$ with $\phi(t)=v^*(t)$ and $\cV(X)=\bbe[v(X)]$.

\subsection{Risk Measures and Connections to Expected Disutility}\label{ss:risk}

A prinicpal contribution of this work is the demonstration of connections
between GVI and risk measures.  Risk measures are functionals used to quantify
the overall loss or hazard associated with a random outcome. The concept of a
risk measure is also closely linked to utility theory.  In fact, given an
expected disutility function $\cV$, we can define the risk measure
$\cR:\mset\to[-\infty,+\infty]$ projected from $\cV$ by
\begin{equation}\label{eq:risk-reg}
  \cR(X) \eqdef \min_{t\in\real} \;\{t + \cV(X-t)\}.
\end{equation}
The class of risk measures projected from expected disutility functions as in
\eqref{eq:risk-reg} form a so-called risk quadrangle
\citep{RTRockafellar_SUryasev_2013a} and are generalizations of optimized
certainty equivalents (OCE)
\citep{ABenTal_MTeboulle_1986a,ABenTal_MTeboulle_2007a}. In particualar,
an OCE is a risk measure given by \eqref{eq:risk-reg} with $\cV(X)=\bbe[v(X)]$,
where $v$ is a disutility function.  It is common to require that a risk
measure $\cR$ satisfies some or all of the following axioms:
\begin{itemize}
  \item[(R1)] {\bf Convexity:} $\cR$ is proper, $\sigma(\mset,\dset)$-lower semicontinuous and convex;
  \item[(R2)] {\bf Monotonicity:} If $X,\,X'\in\mset$ satisfy $X\le X'$ a.s., then $\cR(X)\ge\cR(X')$;
  \item[(R3)] {\bf Translation Equivariance:} $\cR(X+t) = \cR(X)+t$ for all $X\in\mset$ and $t\in\real$;
  \item[(R4)] {\bf Positive Homogeneity:} $\cR(tX) = t\cR(X)$ for all $X\in\mset$ and $t\ge 0$.
\end{itemize}
If $\cR$ satisfies (R1)--(R3), it is referred to as a {\em convex} risk measure
\citep{HFollmer_ASchied_2002a} and if it satisfies (R1)--(R4), it is
{\em coherent} \citep{PArtzner_FDelbaen_JMEber_DHeath_1999a}.  Coherent risk
measures satisfy the worst-case expectation representation
\begin{equation}\label{eq:coherent-risk}
  \cR(X) = \sup_{\eta\in\fA} \;\langle \eta, X\rangle_{\mset^*,\mset},
\end{equation}
where the {\em risk envelope} $\fA\subset\mset^*$ is given by
$\fA\eqdef\mathrm{dom}\,\cR^*=\partial\cR(0)$.  Axiom (R1) ensures that $\fA$ is
nonempty, closed and convex, while (R2) ensures that all $\eta\in\fA$
define nonnegative measures.  In particular, let $B\in\cB$ be such that
$\langle \eta, \mathbbm{1}_B\rangle_{\mset^*,\mset} < 0$, where $\mathbbm{1}_B$
denotes the characteristic function of $B$, i.e., $\mathbbm{1}_B(\theta)=1$
if $\theta\in B$ and $\mathbbm{1}_B(\theta)=0$ if $\theta\not\in B$.  Then,
for any $X\in\mathrm{dom}\,\cR$ and $t\in\real$, we have that
\[
  \begin{aligned}
  \cR^*(\eta) &\ge \sup_{t\in[0,+\infty)}\;\{\langle \eta, X-t\mathbbm{1}_B\rangle_{\mset^*,\mset}
    - \cR(X-t\mathbbm{1}_B)\} \\
  &\ge \sup_{t\in[0,+\infty)}\;\{\langle \eta, X\rangle_{\mset^*,\mset} - \cR(X)
             - t\langle \eta,\mathbbm{1}_B\rangle_{\mset^*,\mset}\} = +\infty.
  \end{aligned}
\]
Axiom (R3) ensures that $\eta\in\fA$ satisfies
$\langle\eta,1\rangle_{\mset^*,\mset}=1$.  To see this, we have that for all
$X\in\mathrm{dom}\,\cR$ and $\eta\in\fA$,
\[
  \begin{aligned}
  \cR^*(\eta) &\ge \sup_{t\in\real}\,\{\langle\eta,X+t\rangle_{\mset^*,\mset} - \cR(X+t)\}\\
              &  = \sup_{t\in\real}\,\{\langle\eta,X\rangle_{\mset^*,\mset} - \cR(X) + t(\langle\eta,1\rangle_{\mset^*,\mset}-1)\}.
  \end{aligned}
\]
Consequently, $\cR^*(\eta) = +\infty$ unless
$\langle\eta,1\rangle_{\mset^*,\mset}=1$.  Hence, the risk envelope $\fA$ is
a subset of probability measures in $\mset^*$. Finally, owing to (R4), $\cR^*$
is the indicator function of $\fA$.  In general, we define the
{\em risk identifiers} $p$ at $X\in\mset$ as the subgradients of $\cR$ at $X$,
i.e., $p\in\partial\cR(X)\subset\fA$.  As a consequence of the Fenchel-Young
inequality, the risk identifiers of a coherent risk measure $\cR$ are the
maximizers on the right-hand side of \eqref{eq:coherent-risk}.

Returning to the generalized OCE \eqref{eq:risk-reg}, it is straightfoward to
show that $\cR$ is a convex risk measure if $\cV$ is an risk-averse expected
disutility function.  Moreover, it is coherent if and only if $\cV$ is
positively homogeneous. To conclude this discussion, we prove fundamental
relations between risk identifiers and subgradients of the regret function
$\cV$ as well as the Fenchel conjugates $\cR^*$ and $\cV^*$.
\begin{lemma}\label{lem:subdiff-risk}
  Let $\cR$ be defined by \eqref{eq:risk-reg}, where $\cV$ is a risk-averse
  expected disutility function, and let $X\in\mset$ be arbitrary.
  Denote by $\bar{t}\in\real$, a solution to the optimization problem defining
  $\cR(X)$, i.e.,
    $\cR(X) = \bar{t} + \cV(X-\bar{t})$.
  Then, we have that
  \[
    \partial\cR(X) = \partial\cV(X-\bar{t})\cap\{p\in\mset^*\,\vert\,\langle p, 1\rangle_{\mset^*,\mset}=1\}.
  \]
\end{lemma}
\begin{proof}
  Let $p\in\partial\cR(X)$, then $\langle p,1\rangle_{\mset^*,\mset}=1$ by (R3)
  and for all $X'\in\mset$ and $t\in\real$, we have that
  \[
    \{t + \cV(X'-t)\} - \{\bar{t}+\cV(X-\bar{t})\}
      \ge \cR(X') - \cR(X) \ge \langle p, (X'-X)\rangle_{\mset^*,\mset}.
  \]
  Consequently, subtracting $(t-\bar{t})$ from both sides yields
  \[
    \cV(X'-t) - \cV(X-\bar{t}) \ge \langle p, (X'-t)-(X-\bar{t})\rangle_{\mset^*,\mset}
    \quad\iff\quad p\in\partial\cV(X-\bar{t}).
  \]
  On the other hand, let $p\in\partial\cV(X-\bar{t})$ satisfy
  $\langle p,1\rangle_{\mset^*,\mset}=1$.  By definition, we have that
  \[
    \cV(X'-t) - \cV(X-\bar{t}) \ge \langle p, (X'-t)-(X-\bar{t})\rangle_{\mset^*,\mset}
      = \langle p,(X'-X)\rangle_{\mset^*,\mset} + (\bar{t}-t)
  \]
  for all $X'\in\mset$ and all $t\in\real$. Subtracting $(\bar{t}-t)$ from both
  sides yields
  \[
    t + \cV(X'-t) - \cR(X) \ge \langle p,(X'-X)\rangle_{\mset^*,\mset}
  \]
  and passing to the infimum over $t\in\real$ on the left-hand side proves the
  desired result.
\end{proof}

\begin{lemma}
  Let $\cR$ be defined by \eqref{eq:risk-reg}, where $\cV$ is a risk-averse
  expected disutility function.  Then, $\cR^*$ is given by
  \[
    \cR^*(\eta) = \left\{\begin{array}{ll}
      \cV^*(\eta)=\Phi(\eta) & \text{if $\langle \eta, 1\rangle_{\mset^*,\mset}=1$} \\
      +\infty & \text{otherwise}
    \end{array}\right. .
  \]
\end{lemma}
\begin{proof}
  Using \eqref{eq:risk-reg}, we arrive at the following equality
  \[
    \begin{aligned}
    \cR^*(\eta) &= \sup_{X\in\mset,\,t\in\real}\;\{\langle \eta, X\rangle_{\mset^*,\mset} - \cR(X-t)\}\\
                &= \sup_{Y\in\mset,\,t\in\real}\;\{\langle \eta, Y\rangle_{\mset^*,\mset}
                     - \cV(Y) + t(\langle \eta, 1\rangle_{\mset^*,\mset}-1)\} \\
                &= \cV^*(\eta) + \sup_{t\in\real}\; t(\langle \eta, 1\rangle_{\mset^*,\mset}-1).
    \end{aligned}
  \]
  The desired result then follows since the supremum in the final equality is
  equal to $+\infty$ unless $\langle\eta,1\rangle_{\mset^*,\mset}=1$.
\end{proof}

\section{Generalized Variational Inference as Risk Identification}\label{sect:genphi}

In this section, we analyze the GVI problem
\begin{equation}\label{eq:gvi0}
  \sup_{p\in\fD}\; \{\bbe[p M_N] - \sigma\Phi(p)\},
\end{equation}
where $\Phi$ is a generalized $\phi$-divergence as described in
Proposition~\ref{prop:Phi-conj}.  In particular, our analysis provides
necessary and sufficient conditions for the existence of solutions to
\eqref{eq:gvi0}, which we refer to as {\em variational densities}, and
transforms the task of solving the optimization problem \eqref{eq:gvi0} into
computing subgradients of a risk measure---a task that is generally simpler
than optimization over probability density functions.  We then introduce a new
GVI problem, which we call coherent GVI (CGVI).  CGVI is closely linked to
coherent risk measures and again can be transformed into the problem of
computing a risk identifer. 

\subsection{Existence and Asymptotics of Variational Densities}

In this subsection, we provide necesary and sufficient conditions for the
existence of solutions to the GVI problem \eqref{eq:gvi0}. In doing so, we
relate the optimization problem \eqref{eq:gvi0} to evaluating a generalized OCE
and as a result of the Fenchel-Young inequality, show that the variational
densities in \eqref{eq:gvi0} are subgradients of that OCE.  The following
result provides this characterization.

\begin{theorem}[Existence of Variational Densities]\label{th:gvi-existence}
  Let the generalized $\phi$-divergence $\Phi$ be defined using the risk-averse
  expected disutility function $\cV$ as in Proposition~\ref{prop:Phi-conj}, and
  let $\cR$ denote the risk measure projected from $\cV$ as in
  \eqref{eq:risk-reg}. Then, the GVI problem \eqref{eq:gvi0} has a solution if
  and only if
  \[
    \partial\cR(\sigma^{-1} M_N)\cap\dset \neq \emptyset
  \]
  and the variational density is a risk identifier of $\sigma^{-1}M_N$, i.e.,
  \[
    \bar{p}\in \partial\cR(\sigma^{-1} M_N)\cap\dset.
  \]
\end{theorem}
\begin{proof}
  We first notice that \eqref{eq:gvi0} is the Fenchel conjugate of
  $\sigma\Phi+\delta_{\fD}$. Consequently, the optimal value in
  \eqref{eq:gvi0} is the infimal convolution of the Fenchel conjugates
  $(\sigma\Phi)^*$ and $\delta_{\fD}^*$.  That is, the optimal value in
  \eqref{eq:gvi0} is given by \begin{equation}\label{eq:gvi1}
    \inf_{Y\in\mset}\; \{(\sigma\Phi)^*(M_N-Y) + \delta_{\fD}^*(Y)\}.
  \end{equation}
  Using the definition of $\Phi$, we have that
  $(\sigma\Phi)^*(Y)=\sigma\cV(\sigma^{-1}Y)\eqdef\cV_\sigma(Y)$.  Moreover, the
  Fenchel conjugate of $\delta_{\fD}$ is the support function of $\fD$, i.e.,
  \[
    \delta_{\fD}^*(Y) = \sigma_{\fD}(Y) = \sup_{p\in\fD} \bbe[p Y] = \esssup Y,
  \]
  where the last equality follows since $\fD$ consists of all probability
  density functions.  Additionally, since $\cV$ is monotonic and $Y\le\esssup Y$
  for all $Y\in\mset$, we have that
  \[
    \cV_\sigma(M_N-Y) \ge \cV_\sigma(M_N-\esssup Y)
    \quad\forall\, Y\in\mset
  \]
  and therefore, the random variable $Y$ in \eqref{eq:gvi1} can be replaced by a
  degenerate random variable (i.e., a scalar).  The optimal value in
  \eqref{eq:gvi1} (and hence \eqref{eq:gvi0}) can thus be equivalently written as
  \begin{equation}\label{eq:gvi2}
    \cR_\sigma(M_N) \eqdef \min_{t\in\real}\; \{ t + \cV_\sigma(M_N-t) \} = \sigma\cR(\sigma^{-1}(M_N-t)),
  \end{equation}
  where $\cR$ is the risk measure defined in \eqref{eq:risk-reg}. As a result
  of the Fenchel-Young inequality, we have that
  \[
    \cR_\sigma(M_N) = \langle p, M_N\rangle_{\mset^*,\mset} - \cR_\sigma^*(p)
      \quad\forall\, p\in\partial\cR_\sigma(M_N).
  \]
  Since $\cR_\sigma(\cdot)=\sigma\cR(\sigma^{-1}\cdot)$, we have
  that $\partial\cR_\sigma(M_N)=\partial\cR(\sigma^{-1}M_N)$ as desired.
\end{proof}
\begin{remark}[Nonexistence]\label{rem:nonexist}
  In general, the estimation problem \eqref{eq:gvi0} need not have a solution
  in $\dset$ since $\partial\cR(\sigma^{-1}M_N)\subseteq\mset^*$.  However,
  Theorem~\ref{th:gen-exist} guarantees that solutions exist when
  $\dset=\mset^*$ and $\partial\cR(\sigma^{-1} M_N)\neq\emptyset$.
\end{remark}

The following two corollaries to Theorem~\ref{th:gvi-existence} provide
sufficient conditions for existence of variational densities to
\eqref{eq:gvi0} when $\dset=\mset^*$ as discussed in Remark~\ref{rem:nonexist}.
\begin{corollary}\label{cor:gvi-existence1}
  In the setting of Theorem~\ref{th:gvi-existence}, if $\dset=\mset^*$ and
  $\cR$ is finite valued and continuous at $\sigma^{-1}M_N$, then
  \eqref{eq:gvi0} has a solution.
\end{corollary}
\begin{proof}
  Under the stated assumptions,
  $\partial\cR(\sigma^{-1}M_N)=\partial\cR(\sigma^{-1}M_N)\cap\dset\neq\emptyset$
  \cite[Prop.~5.2]{IEkeland_RTemam_1999} and the result of
  Theorem~\ref{th:gvi-existence} ensures the existence of solutions.
\end{proof}
\begin{corollary}\label{cor:gvi-existence2}
  In the setting of Theorem~\ref{th:gvi-existence}, if $\dset=\mset^*$ and
  $\cV$ is finite valued, then \eqref{eq:gvi0} has a solution.
\end{corollary}
\begin{proof}
  Since $\cV$ is finite valued, $\cR$ is also finite valued since
  $\cR(Y)\le\cV(Y)$ for all $Y\in\mset$.  Therefore, $\cR$ is continuous on
  $\mset$ \cite[Cor.~2.5]{IEkeland_RTemam_1999}.  The corollary then follows
  from Corollary~\ref{cor:gvi-existence1}. 
\end{proof}

Theorem~\ref{th:gvi-existence} transforms the optimization problem
\eqref{eq:gvi0} into a subdifferentiation problem involving the risk measure
$\cR$ projected from the risk-averse expected disutility function $\cV$.
Giving the connection between $\cR$ and $\cV$, it is possible to rewrite the
condition $p\in\partial\cR(\sigma^{-1}M_N)$ in terms of $\cV$.
\begin{proposition}
  Consider the setting of Theorem~\ref{th:gvi-existence} and let $\bar{t}$ be
  a minimizer in the definition of $\cR_\sigma(M_N)$, i.e.,
  \begin{equation}\label{eq:opt-t}
    \cR_\sigma(M_N) = \min_{t\in\real}\;\{t + \sigma\cV(\sigma^{-1}(M_N-t))\}.
  \end{equation}
  Then, any solution to \eqref{eq:gvi0} satisfies
  \[
    \bar{p}\in \partial\cV(\sigma^{-1}(M_N-\bar{t}))\cap\{p\in\dset\,\vert\,\bbe[p]=1\}.
  \]
\end{proposition}
\begin{proof}
  This result follows directly from Theorem~\ref{th:gvi-existence} and
  Lemma~\ref{lem:subdiff-risk} applied to $\cR_\sigma$.
\end{proof}

The final result in this section concerns the large data limit (i.e., as
$N\to+\infty$) and presents a type of law of large numbers result for the
variational densities $p_N$ associated with the payoff function $M_N$.

\begin{theorem}[Asymptotic Consistency]\label{th:consistency-gvi}
  Consider the setting of Theorem~\ref{th:gvi-existence}.  Let
  $\{\sigma_N\}\subset(0,\infty]$ with $\sigma_N \ge \sigma_0 > 0$ for all $N$
  and $\{M_N\}$ satisfy $M_N\in\mset$ a.s. Further, suppose that there
  exists $p_N\in\partial\cR(\sigma_N^{-1}M_N)\cap\dset$ a.s.\ for all $N$ and
  that $\sigma_N^{-1}M_N\to M_\star$ in $\mset$ a.s., i.e.,
  \[
    \|\sigma_N^{-1}M_N - M_\star\|_{\mset}\to 0 \quad\text{a.s.}
  \]
  Then, any $\sigma(\dset,\mset)$-accumulation point, $p_\star$, of the variational
  densities $\{p_N\}$ satisfies
  \[
    p_\star\in\partial\cR(M_\star)\cap\dset \quad\text{a.s.}
  \]
  In particular, $p_\star$ is a variational density for the asymptotic GVI
  problem
  \[
    \max_{p\in\fD}\; \{\bbe[p M_\star] - \Phi(p)\}.
  \]
\end{theorem}
\begin{proof}
  Outside of a null set, this result follows from
  \cite[Prop.~2.1.5(b)]{FHClarke_1998a}.
\end{proof}
\begin{remark}[Uniform Law of Large Numbers]
  In Theorem~\ref{th:consistency-gvi}, $M_\star$ is typically given by
  $\bbe_{y\sim Y}[m(\cdot,y)]$, where $\bbe_{y\sim Y}$ denotes the expectation
  with respect to the probability law of the noisy data $Y$ and $M_N=m(\cdot,y_N)$
  are generated by identically distributed realizations $y_N$ of $Y$.  In this
  setting, the requirement that $\sigma_N^{-1}M_N\to M_\star$ in $\mset$ a.s.\
  is related to the uniform law of large numbers (LLN).  In particular, if the
  uniform LLN holds and $\sigma_N\to\sigma_\star$, then
  $\sigma_N^{-1}M_N\to M_\star$ also holds. This is the case if, e.g., $\mset$
  is separable, 
  $M_N$ and $m(\cdot,Y)$ are independent and identically distributed $\mset$-valued
  random variables, and $\bbe_{y\sim Y}[\|m(\cdot,y)\|_{\psi^\#}]<\infty$
  \citep[Cor.~3.7.21]{gine2016mathematical}.
  This is also the case if, e.g., $\Theta$ is a compact subset of a Euclidean
  space, $\Upsilon$ is a Euclidean space, $m$ is a Carath\`{e}odory function
  and $m(\theta,\cdot)$ is dominated by an integrable function that is
  independent of $\theta$ \citep[Th.~2]{RIJennrich_1969a}.
\end{remark}

\subsection{Coherent Generalized Variational Inference}

The risk measure $\cR$ projected from $\cV$ as in \eqref{eq:risk-reg}
satisifies (R1)--(R3). However, it need not be coherent unless $\cV$ is
positively homogeneous, in which case $\cV_\sigma=\cV$ and $\cR_\sigma=\cR$.
That is, the learning rate $\sigma^{-1}$ has no effect on the estimation.
This is a desirable property since choosing $\sigma$ can be difficult.
In contrast to requiring $\cV$ to be positively homogeneous, we can modify
$\cR_\sigma$ in the following way.  Given $\varepsilon\ge 0$, we define the
risk measure
\begin{equation}\label{eq:eps-risk}
  \cR^\varepsilon(X) \eqdef \inf_{\lambda > 0} \;\{\lambda\varepsilon + \cR_{\lambda}(X)\}
\end{equation}
The risk measure $\cR^{\varepsilon}$ inherits all properties from $\cR$ with the
addition that it is positively homogeneous, and hence coherent. Since
$\cR^\varepsilon$ is coherent, it satisfies
\begin{equation}\label{eq:coherent-gvi}
  \cR^\varepsilon(X) = \sup_{p\in\fA} \; \bbe[p X],
\end{equation}
where the risk envelope $\fA$ is given by
\begin{equation}\label{eq:riskenv-eps}
  \fA = \fD_\varepsilon \eqdef \{p\in\fD \,\vert\, \Phi(p) \le \varepsilon\},
\end{equation}
where $\Phi = \cV^*$ is the generalized $\phi$-divergence associated with the
regret functional $\cV$.  We refer to the optimization problem on the
right-hand side of \eqref{eq:coherent-gvi} with $X=M_N$ as the coherent GVI
(CGVI) problem.  Intuitively, the CGVI problem seeks a density function $p$
that maximizes the expected payoff $M_N$ within $\varepsilon$ of the prior
distribution, when measured by the generalized $\phi$-divergence $\Phi$.
For example, if $\cV$ is the exponential disutility function, then
$\cR(X)=\log\bbe[\exp(X)]$ is the entropic risk measure and $\cR^\varepsilon$
is the entropic value-at-risk.  Moreover, $\Phi=\cV^*$ generates the KL
divergence.  Consequently, \eqref{eq:coherent-gvi} seeks a probability
density function that maximizes the expected payoff and is within $\varepsilon$
of $\Pi$ when measured by the KL divergence.  An obvious extension is to
replace $\cR^\varepsilon$ with some arbitrary coherent risk measure, not
necessarily of the form \eqref{eq:eps-risk}.  This extension leads to the
interpretation of CGVI as finding a density that maximizes the expected payoff
from the set of admissible densities $\fA$.  However, we will restrict our
attention to \eqref{eq:eps-risk} given it's close ties with GVI.

The CGVI problem presents a trade-off between selecting the learning rate
$\sigma^{-1}$ as in GVI and selecting the $\phi$-diveregence tolerance
$\varepsilon$.  Both tasks can be challenging, however, choosing $\varepsilon$
may be more intuitive in certain applications.  For example, if one can sample
from the Bayesian posterior $p_{\rm post}$, then we can choose
$\varepsilon > \Phi(p_{\rm post})$, ensuring that the CGVI density $\bar{p}$
improves upon the traditional Bayesian solution.  In addition, if an intuitive
method for choosing $\varepsilon$ exists, then the CGVI formulation determines
an optimal inverse learning rate by solving for $\lambda$ in the definition
of $\risk^\varepsilon$.

Our first result proves that the risk envelope is indeed given by
\eqref{eq:riskenv-eps}.
\begin{proposition}\label{prop:riskenv-eps}
  The risk envelope of $\cR^\varepsilon$ for $\varepsilon \ge 0$ is given
  by \eqref{eq:riskenv-eps}.
\end{proposition}
\begin{proof}
  First, notice that for any $p\in\mset^*$, we have
  \[
    (\cR^\varepsilon)^*(p) = \sup_{X\in\mset,\,\lambda > 0} \{\bbe[p X] - \lambda\varepsilon - \cR_\lambda(X)\}
      = \sup_{\lambda > 0} \lambda(\cR^*(p) - \varepsilon).
  \]
  Consequently, $(\cR^\varepsilon)^*(p)$ is finite if and only if
  $\Phi(p) \le \varepsilon$ and $\langle p,1\rangle_{\mset^*,\mset}=1$.
  Finally, $p\in \fA$ satisfies $p\ge 0$ a.s.\ since $\cR^\varepsilon$
  is monotonic as demonstrated in Section~\ref{ss:risk}.
\end{proof}

Proposition~\ref{prop:riskenv-eps} provides the intuitive interpretation of
CGVI: maximize the expected payoff while being with $\varepsilon$ of the
prior.  Analogous to Theorem~\ref{th:gvi-existence}, the following result
provides a necessary and sufficient condition for the existence of variational
densities.

\begin{theorem}[Existence of Variational Densities]\label{th:gen-exist}
  Let the generalized $\phi$-divergence $\Phi$ be defined using the risk-averse
  expected disutility function $\cV$ as in Proposition~\ref{prop:Phi-conj} and
  let $\cR^\varepsilon$ be defined as in \eqref{eq:eps-risk} using the risk
  measure $\cR$ projected from $\cV$ as in \eqref{eq:risk-reg}. Then, the CGVI
  problem is given by
  \begin{equation}\label{eq:cgvi}
    \cR^\varepsilon(M_N)=\sup_{p\in\fD_\varepsilon} \;\bbe[p M_N]
      = \min_{\mu\in\real,\,\lambda \ge 0} \;\{\mu + \lambda\varepsilon + (\lambda\Phi)^*(M_N-\mu)\},
  \end{equation}
  and a variational density exists if and only if
  \[
    \partial\cR^\varepsilon(M_N) \cap \dset \neq \emptyset.
  \]
  Moreover, the variational densities are given by
  \[
    \bar{p}\in\partial\cV(\bar{\lambda}^{-1}(M_N-\bar{\mu}))\cap \{p\in\dset\,\vert\, \bbe[p]=1\},
  \]
  where $(\bar{\mu},\bar{\lambda})$ solve the minimization problem
  on the right-hand side of \eqref{eq:cgvi}.
\end{theorem}
\begin{proof}
  The proof of this fact follows from the Fenchel-Young inequality,
  similar to the proof of Theorem~\ref{th:gvi-existence}.
\end{proof}

As in Remark~\ref{rem:nonexist}, solutions to \eqref{eq:cgvi} may not exist
if $\mset=\dset^*$, even if $\partial^\varepsilon\cR(M_N)\neq\emptyset$.  The
following two corollaries to Theorem~\ref{th:gen-exist} provide sufficient
conditions for existence of variational densities to \eqref{eq:cgvi} when
$\dset=\mset^*$.
\begin{corollary}\label{cor:cgvi-existence1}
  In the setting of Theorem~\ref{th:gen-exist}, if $\dset=\mset^*$ and
  $\cR^\varepsilon$ is finite valued and continuous at $M_N$, then
  \eqref{eq:cgvi} has a solution.
\end{corollary}
\begin{proof}
  The proof of this fact is analogous to the proof of
  Corollary~\ref{cor:gvi-existence1}.
\end{proof}
\begin{corollary}\label{cor:cgvi-existence2}
  In the setting of Theorem~\ref{th:gen-exist}, if $\dset=\mset^*$ and $\cV$ is
  finite valued, then \eqref{eq:gvi0} has a solution.
\end{corollary}
\begin{proof}
  The proof of this fact is analogous to the proof of
  Corollary~\ref{cor:gvi-existence2}.
\end{proof}

The final result in this section concerns the large data limit (i.e., as
$N\to+\infty$) and presents type of law of large numbers for the CGVI densities
$p_N$ associated with the payoff function $M_N$.

\begin{theorem}[Asymptotic Consistency]\label{th:consistency-cgvi}
  Consider the setting of Theorem~\ref{th:gen-exist} and suppose that $\cR$ is
  finite valued on $\mset$. Let $\{\varepsilon_N\}\subset[0,+\infty)$
  and $\{M_N\}$ satisfy $M_N\in\mset$ a.s.  Furthermore, suppose
  that there exists $p_N\in\partial\cR^{\varepsilon_N}(M_N)\cap\dset$ a.s.\ for
  all $N$, that $\varepsilon_N\to\varepsilon\ge 0$ and that $M_N\to M_\star$ in
  $\mset$ a.s.  Then, any $\sigma(\dset,\mset)$-accumulation point, $p_\star$, of the
  variational densities $\{p_N\}$ satisfies
  \[
    p_\star\in\partial\cR^\varepsilon(M_\star)\cap\dset.
  \]
  In particular, $p_\star$ solves the asymptotic CGVI problem
  \[
    \max_{p\in\fD_\varepsilon}\;\bbe[p M_\star].
  \]
\end{theorem}
\begin{proof}
  To prove this result, we first demonstrate that $\cR^{\varepsilon_N}$
  converges to $\cR^\varepsilon$ with respect to the slice topology.  If
  so, then $\partial\cR^{\varepsilon_N}$ graph converges
  $\partial\cR^\varepsilon$ \citep[Th.~4.2]{attouch1993convergence}.
  Suppose this is the case, then for all $X\in\mset$ and
  $\eta\in\partial\cR^\varepsilon(X)$ there exists $(X_N,\eta_N)$ with
  $\eta_N\in\partial\cR^{\varepsilon_N}(X_N)$ for all $N$ such that
  $\eta=\lim_N\eta_N$ in $\mset^*$ and $X=\lim_N X_N$ in $\mset$. Let
  $\{(M_N,p_N)\}$ denote a subsequence (without relabeling) on which
  $p_\star=\operatorname{\sigma(\dset,\mset)-lim}_N p_N$.  Since
  $\cR^{\varepsilon'}$ is convex for all $\varepsilon'\ge 0$,
  $\partial\cR^{\varepsilon'}$ is maximally monotone. Consequently, we have
  that
  \[
    \langle p_N-\eta_N, M_N - X_N\rangle_{\mset^*,\mset} \ge 0 \quad\text{a.s.}\quad\forall\,N.
  \]
  Since $(M_\star-X)=\lim_N (M_N-X_N)$ a.s.\ in $\mset$ and
  $(p_\star-\eta)=\operatorname{\sigma(\dset,\mset)-lim}_N (p_N-\eta_N)$, we
  have that
  \[
    \langle p_\star-\eta, M_\star-X\rangle_{\mset^*,\mset} \ge 0 \quad\text{a.s.}
  \]
  Since this holds for all $(X,\eta)$ with
  $\eta\in\partial\cR^{\varepsilon}(X)$, we have that
  $p_\star\in\partial\cR^\varepsilon(M_\star)$ by the maximal monotonicity
  of $\partial\cR^\varepsilon$, as desired.

  To prove slice convergence, we must show that
  \begin{enumerate}
    \item $\forall\, X\in\mset$, $\exists\,X_N\to X$ in $\mset$ such that $\cR^{\varepsilon_N}(X_N)\to\cR^\varepsilon(X)$;
    \item $\forall\, p\in\mset^*$, $\exists\,p_N\to p$ in $\mset^*$ such that $(\cR^{\varepsilon_N})^*(p_N)\to(\cR^\varepsilon)^*(p)$
  \end{enumerate}
  \citep[Th.~3.1]{attouch1993convergence}.
  We will prove that $\varepsilon\mapsto\cR^\varepsilon(X)$ is continuous for
  fixed $X\in\mset$ by first proving that it is concave. Let
  $\varepsilon_1,\,\varepsilon_2\ge 0$ and $t\in[0,1]$, then
  \[
    \cR^{t\varepsilon_1 + (1-t)\varepsilon_2}(X)
       = \inf_{\lambda>0} \{t (\varepsilon_1\lambda + \cR_\lambda(X)) + (1-t)(\varepsilon_2\lambda + \cR_\lambda(X))\}
     \ge t \cR^{\varepsilon_1}(X) + (1-t)\cR^{\varepsilon_2}(X).
  \]
  Moreover, since
  $\cR^{\varepsilon'}(X) \le \varepsilon' + \cR(X) < +\infty$
  for all $\varepsilon' \ge 0$, $\varepsilon\mapsto\cR^\varepsilon(X)$ is
  finite valued, concave, and hence continuous on $[0,+\infty)$.  Consequently,
  we can take $X_N=X$ for all $N$, which proves the first condition.

  To prove the second, we note that for any $\varepsilon' \ge 0$,
  $(\cR^{\varepsilon'})^*$ is the indicator function of $\fD_{\varepsilon'}$.
  If $\varepsilon=0$, then $\fD_0=\{\pi\}$ and the result is trivial.
  We therefore assume $\varepsilon > 0$.  Let
  $p\in\dset\setminus\fD_\varepsilon$, then
  $\delta_{\fD_{\varepsilon_N}}(p)=+\infty$ for $N$ sufficiently large since
  either $p\not\in\fD$ or $\Phi(p) > \varepsilon$. In this case, we can take
  $p_N=p$ for all $N$ to verify the second condition. On the other hand,
  suppose $p\in\fD_\varepsilon$ and define
  \[
    p_N = \left\{\begin{array}{ll}
      p & \text{if $\varepsilon_N \ge \varepsilon$} \\
      \tfrac{\varepsilon_N}{\varepsilon} p + (1-\tfrac{\varepsilon_N}{\varepsilon}) & \text{if $\varepsilon_N > \varepsilon$}
    \end{array}\right. .
  \]
  Owing to convexity and the fact that $\Phi(1) = 0$, we have
  $\Phi(p_N) \le \varepsilon_N$ for all $N$.  Moreover, we clearly have
  that $p_N\to p$ and $\delta_{\fD_{\varepsilon_N}}(p_N) = 0$ for
  all $N$, proving the second condition.
\end{proof}

\subsection{Application to $\phi$-Divergence-Based Variational Inference}\label{sect:phi-div}
In this subsection, we analyze GVI \eqref{eq:gvi0} and CGVI \eqref{eq:cgvi}
problems when $\Phi$ is a $\phi$-divergence, i.e., $\Phi(p) = \bbe[\phi(p)]$.
Recall that $\phi:\real\to[0,+\infty]$ is a proper, closed and convex function
satisfying $\phi(1)=0$ and $\phi(t) = \infty$ for $t<0$.  Moreover, $\phi$
generates a risk-averse regret function $v=\phi^*$. In this case, the GVI
problem is
\begin{equation}\label{eq:gvi-phi}
  \cR_{\sigma}(M_N) = \sup_{p\in\fD}\; 
    \left\{\bbe\left[p M_N\right] - \sigma\bbe[\phi(p)]\right\}
\end{equation}
and the CGVI problem is
\begin{equation}\label{eq:cgvi-phi}
  \cR^{\varepsilon}(M_N) = \sup_{p\in\fD}\; 
    \left\{\,\bbe[p M_N]
    \,\vert\, \bbe[\phi(p)] \le \varepsilon\,\right\}.
\end{equation}

The natural function spaces in which to analyze \eqref{eq:gvi-phi} and
\eqref{eq:cgvi-phi} are the Orlicz spaces \citep{GAEdgar_LSucheston_1992a}.  To
introduce the Orlicz spaces that we will use, we first recall that the regret
function $v$ associated with $\phi$ is proper, closed, convex, increasing,
and satisfies $v(0) = 0$.  Consequently, $v$ is an Orlicz function as long as
there exists $t >0$ such that $v(t) < +\infty$
\citep[Def.~13.1.1]{rubshtein2016foundations}, which we assume to hold.
We denote the Orlicz conjugate (also called the monotonic conjugate) of
$v$ by $\psi=v^\#$, i.e.,
\[
  \psi(x) \eqdef \sup_{t\ge 0}\{tx - v(t)\}.
\]
A straightforward calculation demonstrates that the Orlicz conjugate of $v$
is given by
\[
  \psi(x) = \left\{\begin{array}{ll}
    0 & \text{if $x<1$} \\
    \phi(x) & \text{if $x\ge 1$} \\
  \end{array}\right. .
\]
and $\psi^\# = v$.  We employ the Orlicz spaces generated by $\psi$ and
$\psi^\#$ to analyze \eqref{eq:gvi-phi} and \eqref{eq:cgvi-phi}.

Recall that the Orlicz space $L_\psi\eqdef L_\psi(\Theta,\cB,\Pi)$ is the linear
space of equivalence classes of $\cB$-measurable functions $X$, equal up to a
set of measure zero, such that $\bbe[\psi(|X/a|)]<\infty$ for some $a>0$ and
that $L_\psi$ is a decomposable Banach lattice when endowed with the Luxemburg
norm
\[
  \|X\|_\psi \eqdef \inf\{\,a>0\,\vert\, \bbe[\psi(|X/a|)] \le 1\,\}.
\]
We similarly define the Orlicz space $L_{\psi^\#}\eqdef L_{\psi^\#}(\Theta,\cB,\Pi)$.
We pair $L_\psi$ with $L_{\psi^\#}$ using the bilinear form
$\langle p, X\rangle\eqdef\bbe[p X]$ for $p\in L_{\psi}$ and $X\in L_{\psi^\#}$,
which is finite by Young's inequality.  In this setting,
$\fD_\varepsilon\subset L_{\psi}$ is bounded and
$\sigma(L_{\psi},L_{\psi^\#})$-closed (see Appendix~\ref{s:ap_Orlicz}).
Here, $\dset=L_{\psi}$ is the space of possible CGVI densities and
$\mset=L_{\psi^\#}$ is the space of pay-off functions $M_N$.  Notice that $\cR$
and $\cR^\varepsilon$ are finite valued on $\mset$ by definition and hence are
continuous and subdifferentiable.  However, this does not guarantee existence
of solutions unless, e.g., $\dset=\mset^*$ is reflexive.
%

Our next results are corollaries of Theorems~\ref{th:gvi-existence} and
\ref{th:gen-exist} that provide a semi-analytic expression for the variational
densities of \eqref{eq:gvi-phi} and \eqref{eq:cgvi-phi} and well as techincal
assumptions on $v$ and $\psi$ that ensure existence of these densities.

\begin{theorem}[Form of $\phi$-Divergence Densities]\label{th:exist-phi}
  Let $\cV(\cdot)=\bbe[v(\cdot)]$, where $v=\phi^*$.  If a variational density
  $\bar{p}\in\dset$ exists for \eqref{eq:gvi-phi}, then it satisfies
  \[
    \bar{p}(\theta) \in \partial v(\sigma^{-1}(M_N(\theta)-\bar{t})),
  \]
  where $\bar{t}$ solves the optimization problem in \eqref{eq:opt-t}.
  On the other hand, if a variational density $\bar{p}\in\dset$ exists
  for \eqref{eq:cgvi-phi}, then it satisfies
  \begin{equation}\label{eq:subdiff}
    \bar{p}(\theta)\in\partial v(\bar{\lambda}^{-1}(M_N(\theta)-\bar{\mu})),
  \end{equation}
  where $(\bar{\lambda},\bar{\mu})$ solves the optimization problem
  in \eqref{eq:cgvi}.
\end{theorem}
\begin{corollary}[Existence of $\phi$-Divergence Densities]\label{cor:exist-phi}
  Consider the setting of Theorem~\ref{th:exist-phi} and suppose $v$
  and $\psi$ are both finite valued and satisfy the ($\Delta_2$)
  condition, i.e., there exists $x_0 > 0$ and $k\in\real$ such that
  \[
    v(2x) < k v(x) \qquad\text{and}\qquad
    \psi(2x) < k \psi(x) \qquad\forall\, x \ge x_0.
  \]
  Then, variational densities exist to \eqref{eq:gvi-phi} and
  \eqref{eq:cgvi-phi}.
\end{corollary}
\begin{proof}
  By construction we have that $\cR$ and $\cR^\varepsilon$ are
  continuous and subdifferentiable.  Moreover, the stated assumptions
  on $v$ and $\psi$ ensure that $\dset=\mset^*$ is reflexive.  Consequently,
  solutions exist by Corollaries~\ref{cor:gvi-existence2} and
  \ref{cor:cgvi-existence2}.
\end{proof}

Corollary~\ref{cor:exist-phi} demonstrates that the GVI problems
\eqref{eq:gvi-phi} and \eqref{eq:cgvi-phi} are well-posed, under certain
assumptions on $v$ and $\psi$, for a fixed set of data $\{y_n\}_{n=1}^N$.
Moreover, Theorem~\ref{th:exist-phi} provides an explicit representation of the
variational densities.  It is important to note that when $v=\phi^*$ is
differentiable and $\bar{\lambda}>0$, these representations simplify to
\[
  \bar{p}(\theta) = v'\left(\frac{M_N(\theta)-\bar{t}}{\sigma}\right)
  \qquad\text{and}\qquad
  \bar{p}(\theta) = v'\left(\frac{M_N(\theta)-\bar{\mu}}{\bar{\lambda}}\right)
\]
for \eqref{eq:gvi-phi} and \eqref{eq:cgvi-phi}, respectively.

\subsection{Examples}
To illustrate the connections between disutility functions, risk measures
and CGVI, we present four common $\phi$-divergence examples.  We note that
the choice of $\phi$ should be left to the practitioner as it is application
dependent and based on their disutility function.  As such, we do not provide
recommendations on how to choose $\phi$, but rather provide discussion for four
common choices of $\phi$. 

\subsubsection{Kullback-Leibler Divergence}\label{ss:KL}

The KL divergence is generated by $\phi(t)=t\log(t)-t+1$ for $t\ge 0$, in which
case $v(t) = e^{t}-1$ is the exponential disutility function
\eqref{eq:exputil} with $r=1$.  See the solid orange line in
Figure~\ref{fig:utility-phi}. For fixed $\sigma > 0$, we can solve the
one-dimensional minimization problem in \eqref{eq:opt-t} with $\cV$
replaced by $\cV_\sigma$, which yields
$\bar{t}=\sigma\log\bbe[\exp(\sigma^{-1}M_N)]$ (and similarly
$\bar{\mu}=\bar{\lambda}\log\bbe[\exp(\bar{\lambda}^{-1}M_N)]$).
The associated risk measure $\cR$ is the entropic risk measure and
$\cR^\varepsilon$ is the entropic value-at-risk \citep{AAhmadiJavid_2012a}.
Given $\sigma > 0$ or the optimal $\bar{\lambda} > 0$ from \eqref{eq:cgvi}, we
can compute the GVI and CGVI densities as
\[
  \bar{p}(\theta)
   = \frac{\exp\left(\sigma^{-1}M_N(\theta)\right)}
  {\bbe\left[\exp\left(\sigma^{-1}M_N\right)\right]}
  \qquad\text{and}\qquad
  \bar{p}(\theta)
   = \frac{\exp\left(\bar{\lambda}^{-1}M_N(\theta)\right)}
  {\bbe\left[\exp\left(\bar{\lambda}^{-1}M_N\right)\right]},
\]
respectively. See Figure~\ref{fig:compare-phi}(a) for an example of the CGVI
density. Note that the GVI and CGVI densities are exactly the Gibbs posterior
with learning rate $\sigma^{-1}=\bar{\lambda}^{-1}$ for CGVI.
Given the relation between KL-divergence-based GVI and the Gibbs posterior, it
is natural to ask if there exists an $\varepsilon$ that produces the Bayesian
posterior when $M_N$ is a log-likelihood function.  The following result
addresses this question.

\begin{proposition}[Optimality of Bayesian Posterior]\label{prop:bayes_eps}
  Suppose the pay-off function $M_N$ is a log-likelihood function
  with $M_N\in L_{\psi^\#}$ and set
  \[
    \varepsilon=D_{\rm KL}(\exp(M_N)/\bbe[\exp(M_N)]\| \pi).
  \]
  Then, the CGVI variational density $\bar{p}$ with $\Phi$ given by
  the KL divergence is the usual Bayesian posterior. In particular,
  $\bar{\lambda} = 1$.
\end{proposition}
\begin{proof}
  This result follows by first differentiating the scalar function
  \[
    \lambda\mapsto \lambda\varepsilon + \lambda
      \log\bbe[\exp(\lambda^{-1}M_N)].
  \]
  Since $\lambda=1$ is strictly interior in the set $[0,\infty)$,
  the desired optimality conditions consist of setting the aforementioned
  derivative to zero and setting $\lambda=1$.  In doing so, we see that
  $\varepsilon$ must be as stated and we recover the Bayesian
  posterior as $\bar{p}$.
\end{proof}

\subsubsection{$\chi^2$ Divergence}\label{ss:chi2}

For the $\chi^2$-divergence, $\phi(t) = (t-1)^2$ for $t\ge 0$ and
$\fD_\varepsilon$ can be equivalently rewritten as
\[
  \fD_\varepsilon = \{\, p\in\fD\,\vert\,
    \bbe[(p-1)^2] \le \varepsilon\,\}
   = \{\, p\in\fD \,\vert\, \bbe[p^2] \le 1+\varepsilon\,\}.
\]
This follows from expanding the quadratic $\bbe[(p-1)^2]$ and noting that
$\bbe[p]=1$.  The set $\fD_\varepsilon$ defines the second-order higher moment
coherent risk measure (see, e.g., \S~8.2 in \cite{PCheridito_TLi_2008a}) with
confidence level $\beta = 1 - (1+\varepsilon)^{-\frac{1}{2}}\in[0,1]$ and
therefore
\begin{equation}\label{eq:HMCR}
  \cR^\varepsilon(M_N) = \min_{a\in\real}\left\{a
    + (1+\varepsilon)^{\frac{1}{2}}
    \bbe[\max\{0,M_N-a\}^2]^{\frac{1}{2}}\right\}.
\end{equation}
Hence, given an optimal $\bar{a}$ from \eqref{eq:HMCR}, we can recover the
optimal density as
\[
  \bar{p}(\theta) = \frac{\max\left\{0,M_N(\theta) - \bar{a}\right\}}
  {\bbe\left[\max\left\{0,M_N - \bar{a}\right\}\right]}.
\]
This follows from \citep[Prop.~3.2]{PCheridito_TLi_2008a}.  See
Figure~\ref{fig:compare-phi}(b) for an example of this density. An interesting
property of the $\chi^2$-divergence (when compared with our other examples) is
that the corresponding CGVI truncates the support of the posterior to account for
$M_N(\theta) \ge \bar{a}$.  Moreover, the associated disutility function for the
$\chi^2$-divergence is the truncated quadratic disutility function
\[
  v(x) = \left\{\begin{array}{ll}
           \tfrac{1}{4}x^2+x & \text{if $x \ge -2$} \\
           -1 & \text{if $x<-2$}
         \end{array}\right.
  \qquad\implies\qquad
  v'(x) = \max\{0,\tfrac{1}{2}x+1\}
\]
and so the GVI density is given by
\[
  \bar{p}(\theta) = \frac{\max\{0,M_N(\theta)-(\bar{t}-2\sigma)\}}{\bbe[\max\{0,M_N-(\bar{t}-2\sigma)\}]}.
\]
See the dashed blue line in Figure~\ref{fig:utility-phi}.

%
%

\subsubsection{R\'{e}nyi Divergence}

The R\'{e}nyi divergence \citep{ARenyi_1961a} is given by
\[
  D_\alpha(p\|\pi)
   \eqdef \frac{1}{\alpha-1}\log\bbe[p^\alpha], \quad\alpha > 0, \;\alpha\neq 1
\]
and $D_\alpha(p\|\pi)$ is the KL divergence when $\alpha = 1$.  The
R\'{e}nyi divergence is a popular alternative to the KL divergence because
as demonstrated in
\citep{JKnoblauch_JJewson_TDamoulas_2019,knoblauch2019frequentist}, it often
outperforms the KL divergence for GVI problems with misspecified priors.
We restrict $\alpha$ to satisfy $0 < \alpha < 1$.
After a sequence of invertible transformations, we can equivalently write
the bound $D_\alpha(p\|\pi)\le\varepsilon$ as
\[
  \frac{1}{\alpha-1}\bbe[p^\alpha - \alpha p +(\alpha-1)]
    \le \frac{1-\exp((\alpha-1)\varepsilon)}{1-\alpha} =:\varepsilon_\alpha,
\]
producing a $\phi$-divergence on the left-hand side with
$\phi(t) = (t^\alpha-\alpha t+(\alpha-1))/(\alpha-1)$ for $t \ge 0$
and $v(t) = (1+t/\beta)^\beta - 1$ where
$\beta=\frac{\alpha}{\alpha-1}<0$.  This $\phi$-divergence is sometimes
called the $\alpha$-divergence.  The associated disutility function is the
isoelastic disutility function \eqref{eq:isoutil}. See the dotted green line in
Figure~\ref{fig:utility-phi}.
Given the optimal $\bar{t},\,\bar{\mu}\in\real$ and $\bar{\lambda} > 0$, we can
then compute the GVI and CGVI densities $\bar{p}$ as
\[
  \bar{p}(\theta)
   = \left(1+\frac{M_N(\theta)-\bar{t}}{\beta\sigma}\right)^{\beta-1}
  \qquad\text{and}\qquad
  \bar{p}(\theta)
   = \left(1+\frac{M_N(\theta)-\bar{\mu}}{\beta\bar{\lambda}}\right)^{\beta-1},
\]
respectively.  See Figure~\ref{fig:compare-phi}(c) for an example of the CGVI
density.


\subsubsection{Total-Variation Distance}\label{s:tv}

For the total variation distance, we set $\phi(t) = |t-1|$ for $t\ge 0$.
As shown in \citep{AShapiro_2017a}, the associated disutility function is
\[
  v(t) = \left\{\begin{array}{ll}
     \max\{0,t+1\} - 1 & \text{if $t\le 1$} \\
     +\infty & \text{if $t > 1$}
  \end{array}\right.
\]
See the dot-dashed yellow line in Figure~\ref{fig:utility-phi}. The Orlicz
spaces generated by $\phi^*$ are $L_{\psi}=L^1(\Theta,\mathfrak{B},\Pi)$ and
$L_{\psi^\#}=L^\infty(\Theta,\mathfrak{B},\Pi)$. As we will see,
total-variation-based GVI and CGVI densities may not exist because the
associated risk identifiers are measures that are not absolutely continuous
with respect to $\Pi$ in general.

We first consider the total-variation-based CGVI problem. Upon the change of
variables $s=\mu+\lambda\in\real$ and $t=\mu-\lambda\in\real$, we can
equivalently rewrite \eqref{eq:cgvi} as
\begin{align}
  \min_{s,\,t\in\real}&\left\{\tfrac{\varepsilon}{2}s
    + (1-\tfrac{\varepsilon}{2})t
    + \bbe[\max\{0,M_N-t\}]\right\}
    \quad\text{subject to}\quad \esssup M_N \le s \nonumber\\
  &=\tfrac{\varepsilon}{2}\esssup M_N
    + \min_{t\in\real}\left\{(1-\tfrac{\varepsilon}{2})t
    + \bbe[\max\{0,M_N-t\}]\right\},\label{eq:tv}
\end{align}
where the optimal $s$ is $\bar{s}=\esssup M_N$ and the optimal
$t$ solves the optimization problem in \eqref{eq:tv}.
Given $\bar{t}$, we recover $\bar{\mu}$ and $\bar{\lambda}$ as
\[
  \bar{\mu} = \tfrac{1}{2}(\esssup M_N+\bar{t})
  \qquad\text{and}\qquad
  \bar{\lambda} = \tfrac{1}{2}(\esssup M_N - \bar{t}).
\]
Notice that if $\varepsilon \ge 2$, then \eqref{eq:tv} is simply
$\esssup M_N$ and the optimal probability measures are convex combinations
of point masses $\delta_{\theta^\star_N}$ centered at solutions
$\theta^\star_N\in\Theta$ to \eqref{eq:intro-0} (here, $\delta_\theta$ denotes
the usual Dirac measure, i.e., for any $S\in\cB$, $\delta_\theta(S) = 1$ if
$\theta\in S$ and $\delta_\theta(S)=0$ if $\theta\not\in S$).  On the other
hand, if $0<\varepsilon < 2$, then \eqref{eq:tv} can be equivalently rewritten
as
\[
  \frac{\varepsilon}{2}\esssup M_N
   + \left(1-\frac{\varepsilon}{2}\right)
     \mathrm{AVaR}_{\varepsilon/2}\left(M_N\right),
\]
where the average value-at-risk (AVaR) is defined by
\[
  \mathrm{AVaR}_\beta(X) \eqdef \frac{1}{1-\beta}\int_\beta^1 q_\alpha(X)\,\mathrm{d}\alpha
    = q_\beta(X) + \frac{1}{1-\beta}\bbe[\max\{0,X-q_\beta(X)\}]
\]
for $0 < \beta < 1$ and $q_\alpha(X)$ is the $\alpha$-quantile of the
random variable $X$ \citep{RTRockafellar_SUryasev_2002a}. In this case, the
optimal $t$ is $\bar{t}=q_{\varepsilon/2}(M_N)$.
Moreover, the subdifferential $\partial\cR^\varepsilon(M_N)$ is
generally set-valued and contains subgradients of the form, e.g.,
\[
  \frac{\varepsilon}{2}\delta_{\theta^\star_N}
    + \left(1-\frac{\varepsilon}{2}\right)P_0,
\]
where $P_0\in\fP$ is absolutely continuous with respect to $\Pi$ and has
density $p_0\in\fD$ satisfying
\[
  p_0(\theta) \in \frac{1}{1-\frac{\varepsilon}{2}}
  \left\{\begin{array}{ll}
    \{0\} & \text{if $M_N(\theta) < \bar{t}$} \\
    \lbrack 0,1\rbrack & \text{if $M_N(\theta) = \bar{t}$} \\
    \{1\} & \text{if $M_N(\theta) > \bar{t}$}
  \end{array}\right..
\]
In particular, if $\Pi(\{\theta\in\Theta\,\vert\,
  M_N(\theta)=q_{\varepsilon/2}(M_N)\})=0$, then
\[
  p_0(\theta) = \frac{1}{1-\frac{\varepsilon}{2}}
    \mathbbm{1}_{[0,\infty)}(M_N(\theta) - \bar{t}).
\]
See Figure~\ref{fig:compare-phi}(d) for an example of this ``density''.
Unfortunately, the relationship between CGVI and AVaR is not present for the
total-variation-based GVI problem \eqref{eq:gvi-phi}.  In this setting, we
can make the substitution $s=t+\sigma$ in \eqref{eq:opt-t} to rewrite the
GVI problem as
\[
  \min_{s\in\real} \; \{ s + \bbe[\max\{0,M_N-s\}]\}
  \qquad\text{subject to}\qquad
  \esssup M_N - 2\sigma \le s.
\]
Clearly, if $\tfrac{1}{2}(\esssup M_N - \essinf M_N) \le \sigma$, then we can
set $\bar{s}=\essinf M_N$ and $\cR_\sigma(M_N) = \bbe[M_N]$.  In this case, the
variational density is $\bar{p}\equiv 1$. The more interesting case occurs when
$\tfrac{1}{2}(\esssup M_N - \essinf M_N) > \sigma$, which implies
$\essinf M_N < \esssup M_N-2\sigma\le \bar{s}\le \esssup M_N$. Consequently,
$\bbe[M_N] \le \cR_\sigma(M_N) \le \esssup M_N$.  Moreover, we have that
$\sigma\mapsto\cR_\sigma(M_N)$ is monotonically decreasing, attaining its
lower bound when $\sigma=\tfrac{1}{2}(\esssup M_N - \essinf M_N)$ and
its upper bound when $\sigma = 0$.


\begin{figure}[!ht]
  \centering
  \includegraphics[width=\textwidth]{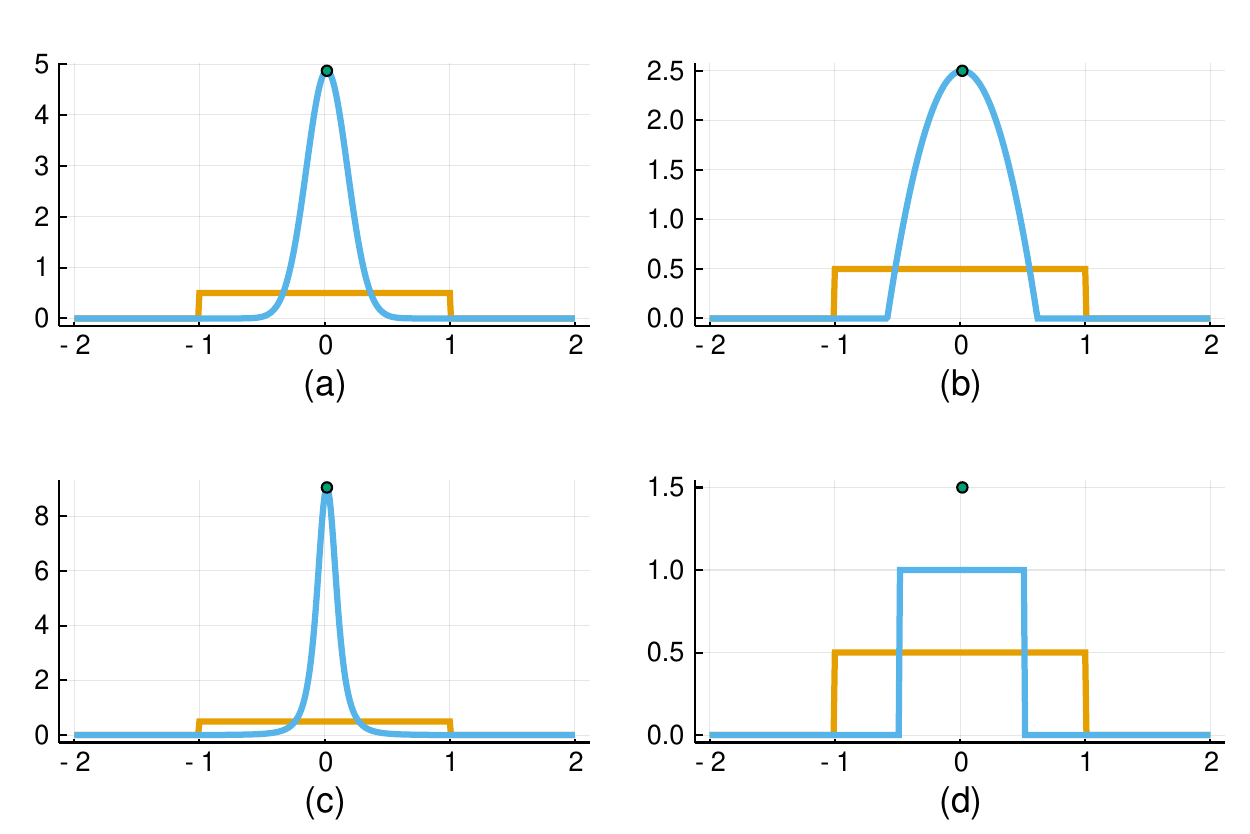}
  \caption{$\phi$-divergence CGVI densities for (a) KL, (b) $\chi^2$,
    (c) R\'{e}nyi ($\alpha=\tfrac{1}{2}$), and (d) total variation.  In
    this example, $\Upsilon=\Theta=\real$, 100 realizations of $Y$ were
    drawn from a standard normal distribution, $\Pi$
    is the uniform measure on $[-1,1]$, $m$ is the standard normal
    log-likelihood function, and $\varepsilon=1$.}
  \label{fig:compare-phi}
\end{figure}

\section{Computation, Sampling and Asymptotic Analysis}\label{sec:comp}

In this section, we investigate several theoretical properties that are
relevant for the computation of the $\phi$-divergence-based CGVI variational
densities analyzed in Subsection~\ref{sect:phi-div} and note that it may be
possible to extend these results to general $\Phi$ using empirical estimates
for law-invariant functionals
(cf.~\citep[Ch.~7.2.6]{AShapiro_DDentcheva_ARuszczynski_2014a}). In practice,
we can only expect to solve the one-dimensional GVI \eqref{eq:gvi-phi} and
two-dimensional CGVI \eqref{eq:cgvi-phi} problems approximately by either
stochastic approximation algorithms or by empirical approximation using samples
from the prior distribution $\Pi$. The latter is often referred to as sample
average approximation (SAA) in stochastic programming. This raises important
questions regarding the asymptotic behavior of the sample-based solutions as
the sample size increases to infinity.

Given $S$ independent and identically distributed (iid) samples
$\{\theta_s\}_{s=1}^S$ drawn from $\Pi$, we evaluate the pay-off samples
$\{M_N^s\eqdef M_N(\theta_s)\}_{s=1}^S$ and then solve the SAA problem
\begin{equation}\label{eq:SAA}
  \widehat{\cR}^{\varepsilon}_{S}(M_N) \eqdef \inf_{\mu\in\real,\;\lambda\ge 0}\left\{\,
     \mu + \lambda\varepsilon
         + \frac{1}{S}\sum_{s=1}^S(\lambda\phi)^*(M_N^s - \mu)\right\}.
\end{equation}
We denote the set of minimizers to \eqref{eq:SAA} by $\widehat{X}^{\star}_{S}$
and the minimizers to \eqref{eq:cgvi-phi} by $X^\star$.
The optimization problem \eqref{eq:SAA} is a two-dimensional convex
optimization problem that can be solved using any convex
programming methods.  For example, one can employ proximal or projected
(sub)gradient-type methods.  If $\phi^*$ is sufficiently differentiable, one
can employ Newton-type methods.  See \citep{ABeck_2017} for a survey of other
applicable first-order methods. Additionally, under appropriate assumptions,
one can prove that the estimators $\mu_S$ and $\lambda_S$ computed by solving
\eqref{eq:SAA} converge a.s.\ to a solution of \eqref{eq:cgvi-phi}, see e.g.,
\citep{JDupacova_RJBWets_1988a,AJKing_RTRockafellar_1993a,AShapiro_1989a} and
the result below. Moreover, when $\bar{\lambda}>0$,
\citep[Prop.~2.1.5(b)]{FHClarke_1983} ensures that any
$\sigma(L_\psi,L_{\psi^\#})$-accumulation point $\bar{p}$ of the sequence
$\{p_S\}$ with
$
  p_S(\theta) \in \lambda_{S}\partial\phi^*((M_N(\theta) - \mu_S)/\lambda_S),
$
satisfies \eqref{eq:subdiff} a.s. 

We emphasize here that there are no limitations on the dimensions of $\Theta$.
In particular, the size of $\theta$ does not change the computational
complexity of solving the two-dimensional convex optimization problem
\eqref{eq:SAA} since the pay-off samples $M_N(\theta_s)$ are computed offline,
prior to solving \eqref{eq:SAA}. However, it is important to note that the
number of samples $S$ required to achieve a prescribed accuracy when solving
the SAA problem \eqref{eq:SAA} is typically dimension dependent
\citep[Chap.~5]{AShapiro_DDentcheva_ARuszczynski_2014a}. Once $\bar{\lambda}$
and $\bar{\mu}$ are computed, we can generate samples from the CGVI
distribution $\bar{p}$ using similar methods as those used in Bayesian
inference such as Markov Chain Monte Carlo (MCMC)
\citep{ABeskos_MGirolami_SLan_PEFarrell_AMStuart_2017} and transport maps
\citep{TAELMoselhy_YMMarzouk_2012}.  We also note that for many choices of
$\phi$, the density $\bar{p}$ is given in closed form once $\bar{\lambda}$
and $\bar{\mu}$ are computed.  This allows us to compute various statistics
such as the value of $\theta$ that maximizes $\bar{p}$.

In the next result, we prove asymptotic consistency of the optimal values and
optimal solution set of  \eqref{eq:SAA} in the large sample limit (i.e.,
$S\to\infty$). These statements follow by verifying the assumptions of
\citep[Th.~5.4]{AShapiro_DDentcheva_ARuszczynski_2014a}.  We prove the various
necessary conditions in Appendix~\ref{ap:consistency}
and then obtain the asymptotic consistency as a corollary.
We refer the reader to \citep[Chap.~14]{RTRockafellar_RJBWets_1998a} or
\citep[Chap.~7]{AShapiro_DDentcheva_ARuszczynski_2014a} regarding the
terminology.
\begin{theorem}[Asymptotic Consistency: Sample Limit]\label{cor:asymp}
If the asumptions of Corollary~\ref{th:exist-phi} hold, then
\begin{enumerate}
\item $\widehat{\cR}^{\varepsilon}_{S}(M_N) \to \cR^{\varepsilon}(M_N)$ a.s.;
\item The deviation of $\widehat{X}^{\star}_{S}$ to $X^{\star}$ converges
      to zero a.s. That is, we have that
 \[
   \mathbb D(\widehat{X}^{\star}_{S},X^{\star}) \eqdef \sup_{(\mu,\lambda)
   \in \widehat{X}^{\star}_{S}} \inf_{(\mu',\lambda')
   \in X^{\star}} \| (\mu,\lambda) - (\mu',\lambda')\| \to 0 \quad \text{a.s.}
 \]
\end{enumerate}
\end{theorem}
\begin{proof}
This is a direct application of
\citep[Th.~5.4]{AShapiro_DDentcheva_ARuszczynski_2014a} in light of
Theorem~\ref{thm:asymp}.
\end{proof}
Theorems~\ref{thm:asymp}~and~\ref{cor:asymp} lead to the following
observations. As proven in Theorem~\ref{thm:asymp}, the solution sets $X^\star$
and $\widehat{X}^\star_S$ are always nonempty,
compact, and convex subsets of $\real \times [0,\infty)$, regardless of the
sample size $S$. For any $(\mu_S,\lambda_S) \in \widehat{X}^\star_{S}$, there
exists $(\mu_S',\lambda_S')\in X^\star$ satisfying
\[
 \| (\mu_S,\lambda_S)  - (\mu'_S,\lambda'_S)\|
 = \min_{(\mu',\lambda') \in X^{\star}} \| (\mu_S,\lambda_S)  - (\mu',\lambda')\|
 \le \mathbb D(\widehat{X}^\star_S,X^\star)
\]
since $X^\star$ is compact. Using again the compactness of $X^\star$, we can
deduce the existence of some $(\bar{\mu},\bar{\lambda})\in X^\star$ and a
subsequence $\{(\mu'_{S_{k}},\lambda'_{S_{k}})\}$ that converges to
$(\bar{\mu},\bar{\lambda})$ a.s. It follows that
\[
 \| (\mu_{S_k},\lambda_{S_k})  -(\bar{\mu},\bar{\lambda})\| \le 
 \mathbb D(\widehat{X}^\star_{S_k},X^\star) + \|(\mu'_{S_{k}},\lambda'_{S_{k}}) - (\bar{\mu},\bar{\lambda}) \|
\]
and, by Theorem~\ref{cor:asymp}, we have that
\[
 \| (\mu_{S_k},\lambda_{S_k})  -(\bar{\mu},\bar{\lambda})\|  \to 0 \quad\text{a.s.}
\]
Consequently, we have that
\[
 \mu_{S_k} + \varepsilon\lambda_{S_k}+\frac{1}{S_k}\sum_{k=1}^{S_k}(\lambda_{S_k}\phi)^*(M_N^k - \mu_{S_k}) 
    \to \bar{\mu} + \varepsilon\bar{\lambda} + \bbe[(\bar{\lambda}\phi)^*(M_N - \bar{\mu})] \quad\text{a.s.}.
\]

In the following corollary, we use the above observations to relate the
solutions $(\mu_{S},\lambda_{S})$ to sequences of CGVI variational densities.
\begin{theorem}
  Let the assumptions of Corollary~\ref{th:exist-phi} hold and suppose $L_\psi$
  is reflexive.  If $\{\bar{p}_k\}$ denotes a sequence of SAA estimators to the
  CGVI posterior associated with $\{S_k\}$ samples with $S_k\nearrow\infty$,
  i.e.,
  \[
    \bar{p}_k(\theta)\in\partial (\lambda_{S_k}\phi)^*(M_N(\theta)-\mu_{S_k}).
  \]
  Then any a.s.\ $\sigma(L_\psi,L_{\psi^\#})$-accumulation point
  $\bar{p}$ of $\{\bar{p}_k\}$ is a variational density for
  \eqref{eq:cgvi-phi}, i.e., $\bar{p}$ satisfies \eqref{eq:subdiff}.
\end{theorem}
\begin{proof}
This proof uses various notions of variational convergence.  See
\citep{HAttouch_1984a} for an overview of these techniques.
One can show that the sequence of functionals 
$\Phi_k(p) \eqdef \lambda_{S_k}\bbe[\phi(p)]$ pointwise and epi-converge
to $\overline{\Phi} \eqdef \overline{\lambda}\bbe[\phi(\cdot)]$ a.s.\
\citep[Sect.~7.2.5]{AShapiro_DDentcheva_ARuszczynski_2014a}. 
Moreover, since $\langle \cdot,X \rangle : L_{\psi} \to \real$ is a
continuous linear functional for any $X \in L_{\psi^{\#}}$,
the sequence $\Phi_k(\cdot) - \langle \cdot,X\rangle$ also pointwise
and epi-converge a.s. Consequently, we have that
\[
\inf_{p \in L_{\psi}}\left\{ \Phi_k(p) - \langle p, X \rangle \right\} \to 
\inf_{p \in L_{\psi}}\left\{ \overline{\Phi}(p) - \langle p, X \rangle \right\}
\quad\text{a.s.}
\]
\citep{HAttouch_1984a}.
As a result, $(\Phi_k)^*$ converges to $(\overline{\Phi})^*$ pointwise a.s. In
fact, if $X_k \to X$ with respect to the
$\sigma(L_{\psi^{\#}},L_{\psi})$-topology, then passing to the limit inferior
(a.s.) on both sides and taking the supremum over $p$ in the inequality
\[
-\inf_{p \in L_{\psi}}\left\{ \Phi_k(p) - \langle p, X_k \rangle \right\} \ge
-( \Phi_k(p) - \langle p, X_k \rangle)\quad \forall p \in L_{\psi} \quad\text{a.s.}
\]
yields
\[
\liminf_{k\to\infty} (\Phi_k)^*(X_k) \ge (\overline{\Phi})^*(X) \quad\text{a.s.},
\]
since the pointwise and epi-limits of $\Phi_k$ coincide. These facts ensure
that $(\Phi_k)^*$ Mosco converges to $(\overline{\Phi})^*$ and therefore the
subdifferentials $\partial (\Phi_k)^*$ graph converge to
$\partial(\overline{\Phi})^*$ by \citep[Th.~3.66]{HAttouch_1984a}.
Lastly, we note that $M_N - \mu_{S_k}$ strongly converges in $L_{\psi^{\#}}$
a.s. 
Therefore, if $\overline{p}$ is an a.s.
$\sigma(L_{\psi},L_{\psi^\#})$-accumulation point of $\left\{\bar{p}_k\right\}$
with $\overline{p}_k \in \partial (\Phi_k)^*(M_N - \mu_{S_k})$, then
$\overline{p} \in \partial (\overline{\Phi})^*(M_N - \bar{\mu})$ and by 
\citep[Th.~21]{RTRockafellar_1973}, we have that
\[
  \bar{p}(\theta) \in \partial (\bar{\lambda} \phi)^*(M_N(\theta) - \bar{\mu})
  \quad\text{a.s.}
\]
Due to Theorem~\ref{cor:asymp}, the discussion above surrounding the derivation
of $(\mu_{S_k},\lambda_{S_k})$, and strong duality, $\bar{p}$ is optimal for
\eqref{eq:MLE}. 
\end{proof}

\subsection{Performance of CGVI under Model Misspecification}\label{ssec:misspec}
As noted earlier, using other $\phi$-divergences such as the R\'{e}nyi
divergence in place of the KL divergence may produce better estimates
of the parameters $\theta$.  We demonstrate this feature with a small example
in which both the prior and model are misspecified and only a small number of
samples is available. 

Given the data $\left\{y_n=(x_n,z_n)\right\}_{n=1}^N \subset \mathbb R^2$
with $N=100$, we postulate a linear decision function $h_\theta(x) = \theta x$
and make the assumption that the observation errors
$\epsilon_n\eqdef z_n-h_{\theta^{\star}}(x_n)$ are iid standard normal
random variables. For the experiment, however, we generate the data using
the heteroscedastic model
\[
  z_n = x_n + \sigma(x_n) \eta_n
\]
(i.e., $\theta^\star=1$), where $\eta_n$ are iid standard normal random
variables and
\[
\sigma(x) = \left\{
\begin{array}{ll}
0.04   & \text{if $x <  \xi^N_{5}$} \\
 0.4   & \text{if $x \in [\xi^N_{5},\xi^N_{95}]$} \\
1.0    & \text{otherwise}
\end{array}
\right. .
\]
Here, $\xi^N_{5}$ and $\xi^N_{95}$ denote the 5\% and 95\% quantiles of the
standard normal distribution, respectively. Using the standard normal error
assumption, the log-likelihood function is
\[
  m(\theta,y_n) = -\frac{1}{2}| z_n - \theta x_n |^2.
\]
For the prior, we choose $\Pi$ to be normally distributed with mean -2 and
unit variance. We note that the true parameter value $\theta^\star=1$ is in the
support of $\Pi$, but it is more than two standard deviations from the mean.  

For the $\phi$-divergence CGVI problems, we set $\varepsilon$ to be twice the
$\phi$-divergence of the Bayesian posterior (BS) from the prior $\Pi$.  We
compute these values using $S=10,000$ samples of $\theta$ drawn from the
prior for the KL, the $\chi^2$ (CS), and the R\'{e}nyi (RY) with $\alpha=0.5$
divergences.  We list the values of $\varepsilon$ in Table~\ref{tbl:misspec}.
\begin{table}[!ht]
\centering
\begin{tabular}{l | c c c c c}
  & BS & KL & CS & RY & TV \\
  \hline
  $\varepsilon$ & ---      & 1.991814 & 5.842229 & 1.106317 & 1.376138 \\
  $\eta$        & 1.632968 & 1.092635 & 1.579035 & 0.367454 & 0.108191
\end{tabular}
\caption{The constraint tolerance $\varepsilon$ and parameter errors $\eta$ (to
  six significant digits) for the Bayesian postior (BS), the Kullback-Leibler
  divergence (KL), the $\chi^2$ divergence (CS), the R\'{e}nyi divergence (RY),
  and the total variation distance (TV).}
\label{tbl:misspec}
\end{table}
This choice of $\varepsilon$ ensures that the Bayesian posterior
density is feasible for each of the CGVI problems.
We then compute the optimal $\lambda_S$ and $\mu_S$ for each
example. For KL and CS, we compute $\lambda_S$ and $\mu_S$,
respectively, using Ridder's root finding algorithm. For RY, we compute the
optimal parameters using a projected gradient method. We plot these
densities in the left image of Figure~\ref{fig:eme_phidiv_densities} along
with their value at the true parameter $\theta^\star=1$. In each case, this
value is slightly skewed to the right of the mode. In addition, we plot the
full posterior densities in the right image of
Figure~\ref{fig:eme_phidiv_densities}.

For the total variation CGVI (TV), we proceed in a different manner. This is due
to the fact that TV is fundamentally different than the other
$\phi$-divergence examples as it includes point masses at the maximizers of
$M_N$, i.e., the maximum likelihood estimators. For this one-dimensional
example, these are easy to estimate. For larger dimensional problems,
one would require a robust optimization method to solve the associated
nonlinear program. To compute $\varepsilon$, we first compute the TV
distance of the Bayesian posterior from the prior using a grid on $[-20,20]$
of size $S = 100,000$. This value is listed in the rightmost column of
Table~\ref{tbl:misspec}.
Scaling this value by two would yield a number larger than two. Consequently,
the CGVI would only be comprised of convex combinations of point masses at the
maximum likelihood estimators. Instead, we set $\varepsilon$ as in
Table~\ref{tbl:misspec} and compute the $\varepsilon/2$ quantile, which we use
to approximate the cumulative distribution function (cdf) of the full TV
posterior, see Figure~\ref{fig:TV_cdf}.

Finally, we judge the performance of the various CGVIs by estimating the mode of
the posteriors and comparing them in absolute value to the true parameter value
$\theta^\star = 1$.  We denote this error by $\eta$ and list the values in
the second row of Table~\ref{tbl:misspec}.
As expected, the standard Bayesian posterior yields the worst estimate. This is
followed by CS, despite exhibiting a smaller variances than the other
CGVIs (see Figure~\ref{fig:eme_phidiv_densities}).  KL outperforms
both of these. However, the clear winner amongst the CGVIs, other than TV,
is RY, which corroborates the observations in
\citep{JKnoblauch_JJewson_TDamoulas_2019}. Finally, we see that TV is
the least susceptible to small data and misspecification as it provides the
best estimate. However, it is more limited in applicability than RY as one must
compute the maximum likelihood estimators and the quantiles of $M_N$.

\begin{figure}[!ht]
  \centering
   \includegraphics[width=0.43\textwidth]{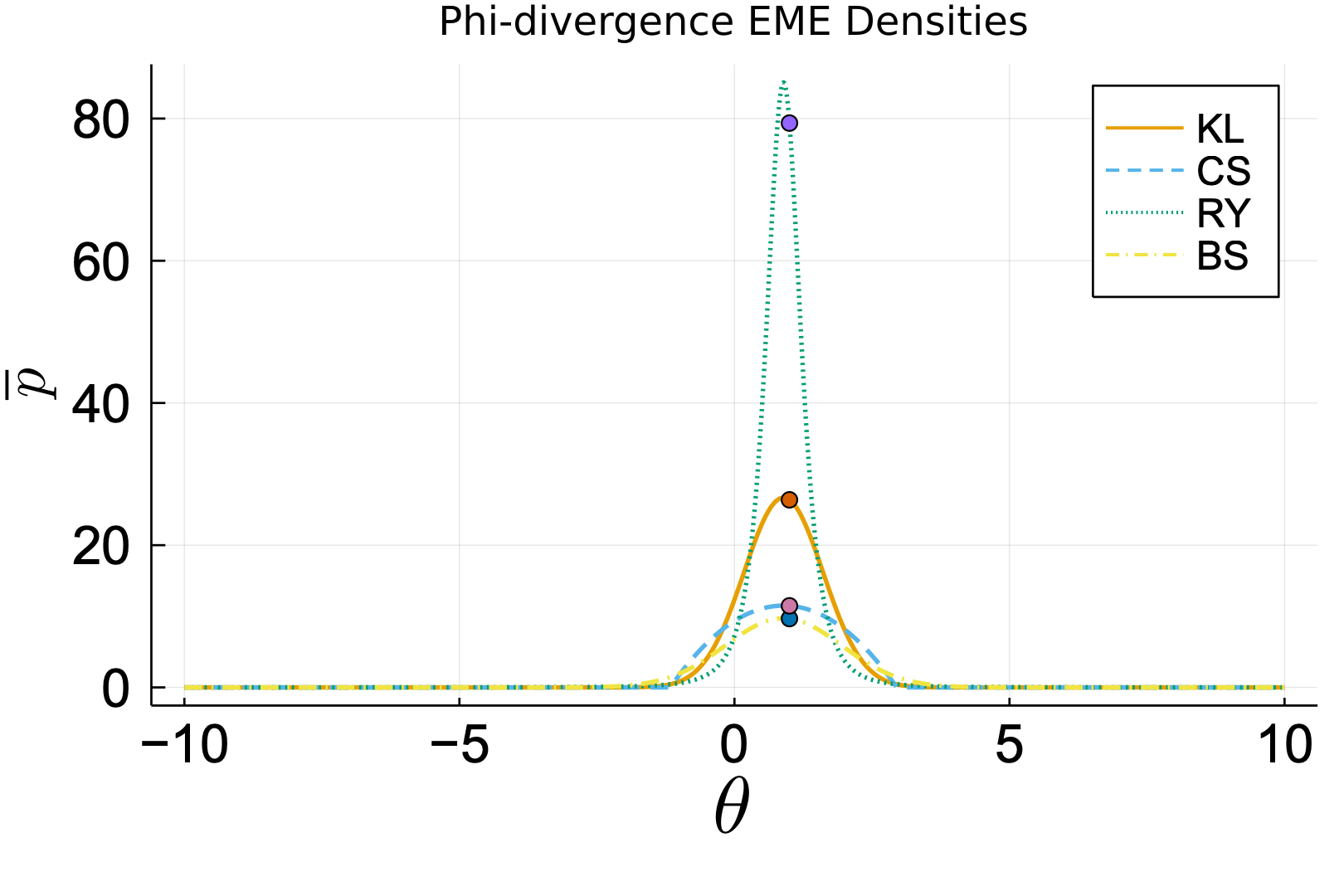}
    \includegraphics[width=0.43\textwidth]{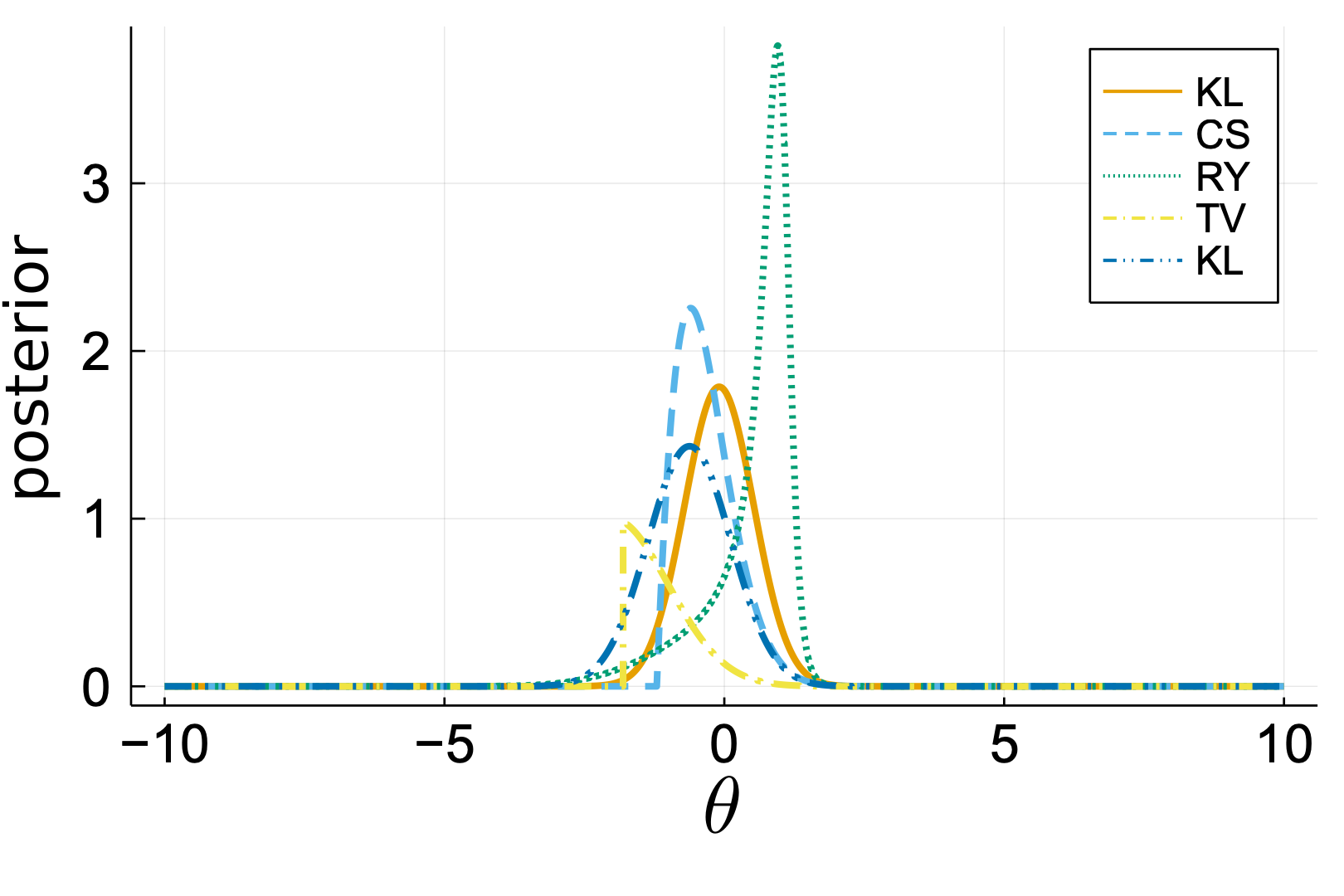}
  \caption{Left: CGVI densities for standard Bayes, KL, $\chi^2$ and Renyi ($\alpha = 0.5$) divergences. Right: full posterior densities for standard Bayes, KL, $\chi^2$ and Renyi ($\alpha = 0.5$)
  }
  \label{fig:eme_phidiv_densities}
\end{figure}


\begin{figure}[!ht]
  \centering
   \includegraphics[width=0.43\textwidth]{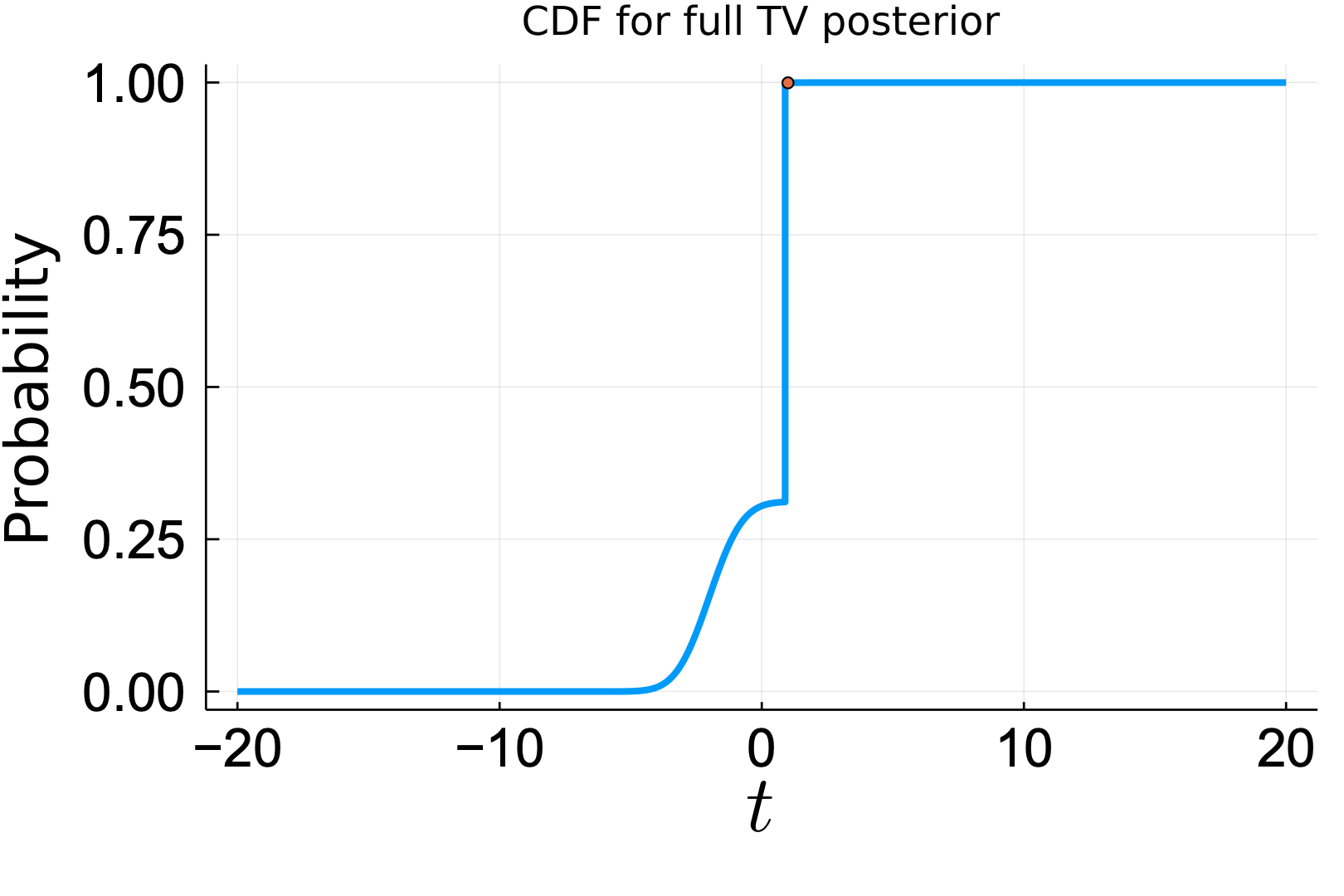}
  \caption{An empirical estimate of the cumulative distribution function for the TV posterior. The jump occurs at the estimated maximum likelihood estimator.
  }
  \label{fig:TV_cdf}
\end{figure}

\section{Empirical CGVI: An ``Objective'' Perspective}\label{sec:empir}

%
In this brief section, we present several extensions to CGVI in which we assume that no reasonable prior measure $\Pi$ exists. Instead, we wish to proceed in a more objective manner in which perhaps only samples of the prior or moments are available and we are unsure which utility function to choose. We introduce two approaches to handle this situation. The first approach falls into the class of problems described by \eqref{eq:div} whereas the second approach is based on moment matching. From a computational standpoint, more regularity of $M_N$ and the integrands $\psi$ in the case of moment-matching would be necessary to compute the global minimizers. A deeper look into the computation of empirical CGVIs will be the subject of future investigations.

\subsection{Wasserstein Distance and Empirical Priors}

In this subsection, we assume that we have many noisy observations of $Y$ and
a few noisy realizations of $\theta$.  For example, these realizations could be
the result of prior estimation attempts.  This could also model the practical
situation in which observations of $Y$ are cheap to obtain relative to
observations of $\theta$.  To formulate the estimation problem \eqref{eq:MLE},
we resort to empirical estimation to generate the prior distribution $\Pi$.

Let $\Theta$ be a locally $\sigma$-compact Polish space endowed with its
Borel $\sigma$-algebra $\cB$.
According to \citep[Th.~1]{RWilliamson_JLudvik_1987}, $\Theta$ admits a metric 
$d: \Theta \times \Theta \to \mathbb R$ that has the Heine-Borel property
(i.e., all closed and bounded sets are compact). One obvious example for such
a $\Theta$ is $\real^n$ with the usual Euclidean topology.  We further require
that $M_N$ is bounded, continuous, and concave on $\Theta$.
Given $M$ observations of $\theta$, denoted by
$\{\widehat{\theta}_m\}_{m=1}^M$, we define the prior measure $\Pi$ as the
empirical measure:
\[
  \Pi = \frac{1}{M}\sum_{m=1}^M \delta_{\widehat{\theta}_m},
\]

One attractive approach to defining a meaningful set of measures 
$\fA$ in \eqref{eq:MLE} is to use probability metrics, cf.\
\citep{STRachev_1991}.  To this end, we employ the Wasserstein-1 distance (i.e.,
the Kantorovich-Rubinstein metric), $W_1:\fP\times\fP\to[0,\infty]$, defined by
\begin{align*}
  W_1(P,Q) \eqdef \inf_{\rho \in \Gamma(P,Q)}
  \int_{\Theta^2} d(\theta_1,\theta_2)\,
  \rho(\mathrm{d}\theta_1,\mathrm{d}\theta_2),
\end{align*}
where $\Gamma(P,Q)\subset\fP(\Theta\times\Theta)$ denotes the subset of all
Borel probability measures on $\Theta\times\Theta$ that have marginal $P$ for
their first factor and marginal $Q$ for their second.  We then define the set
$\fA$ by
\[
  \fA = \{\, P\in\fP\,\vert\, W_1(P,\Pi) \le \varepsilon\,\}.
\]
For this choice of $\fA$, we have the following existence result.
\begin{theorem}[Existence of Wasserstein CGVI]\label{prop:wasserstein}
  Under the stated assumptions, \eqref{eq:MLE} admits a solution $P^{\star}$
  for any $\varepsilon > 0$.
\end{theorem}
\begin{proof}
Since $M_N$ is bounded and continuous, $\bbe_P\left[M_N\right]$ is
continuous in $P \in \mathfrak{P}(\Theta)$ with respect to the weak-convergence
of probability measures. It follows from \citep[Prop.~3]{APichler_HXu_2019} that
$\fA$ is weakly compact. In order to use the proof of
\citep[Prop.~3]{APichler_HXu_2019}, we first note that
$\mathcal{P}\eqdef\{\Pi\}$ is a uniformly tight set of probability measures.
Moreover, according to \citep[Th.~1]{RWilliamson_JLudvik_1987}, given a compact
set $C \subset \Theta$ and some constant $\delta > 0$, the set
$
  C_{\delta}\eqdef\left\{\theta\in\Theta\,\vert\, d(\theta,c)\le\delta,
    c\in C \right\}
$
is a closed and bounded set and therefore, compact. The rest of the proof in
\citep[Prop.~3]{APichler_HXu_2019} can be used without change. Finally, the
existence of $P^{\star}$ for \eqref{eq:MLE} now follows by the usual arguments
from the direct method of the calculus of variations, cf.,
\citep[Th.~3.2.6]{HAttouch_GButtazzo_GMichaille_2006a}.
\end{proof}

\begin{remark}[Assumptions in Theorem~\ref{prop:wasserstein}]
The assumptions of Theorem~\ref{prop:wasserstein} are restrictive in the sense
that it does not appear possible to extend the current proof to a non-trivial
infinite-dimensional setting. On the other hand, $\Theta$ is still allowed to
be of arbitrarily high dimension, which is clearly of interest to modern
applications in machine learning and data science. 
\end{remark}

In light of Theorem~\ref{prop:wasserstein}, we know that there exists a
$P^{\star} \in \fA$ such that
\[
\bbe_{P^{\star}}[M_N] = \max_{P \in \fA} \bbe_{P}[M_N].
\]
To compute the CGVI measure $P^\star$, we can reformulate the estimation
problem \eqref{eq:MLE} as a finite-dimensional optimization problem
(cf.\ \citep{PMEsfahani_DKuhn_2018a}).  In
particular, the optimal value in \eqref{eq:MLE} is equal to the optimal value
of the problem
%
%
\begin{subequations}\label{eq:wass}
\begin{align}
  &\sup_{\theta_1,\ldots,\theta_M\in\Theta}
    \,\frac{1}{M}\sum_{m=1}^M M_N(\theta_m) \\
  &\operatorname*{subject\;to}\quad
    \frac{1}{M}\sum_{m=1}^M d(\theta_m,\widehat{\theta}_m)\le\varepsilon. 
\end{align}
\end{subequations}
The arguments in the proofs of Theorems~4.2 and 4.4 of
\citep{PMEsfahani_DKuhn_2018a} extend directly to our more general setting.
Therefore, if $\{\theta_1^\star,\ldots,\theta_M^\star\}$ is a solution to
\eqref{eq:wass}, then the optimal measure is given by
\[
  P^\star = \frac{1}{M}\sum_{m=1}^M \delta_{\theta^\star_m}.
\]

\subsection{Moment Matching}

For this approach, we assume that we are given noisy observations,
$\mathfrak{m}_m$, of auxiliary quantities with the form $\bbe[\psi_m]$ where
$\psi_m:\Theta\to\real$ for $m=1,\ldots,M$ are $\cB$-measurable functions that
are $\Pi$-integrable, where $\Pi\in\fP$ again is a predetermined prior measure.
In this case, we choose $\fA$ to be the subset of probability measures that
match the computed generalized moments $\mathfrak{m}_m$ up to the fixed
tolerances $\varepsilon_m\ge 0$, i.e.,
\[
  \fA = \{\,P\in\fP\,\vert\,
    |\bbe_P[\psi_m]-\mathfrak{m}_m|
    \le\varepsilon_m,\;\;m=1,\ldots,M\,\}.
\] 
Owing to a result from Rogosinski \citep{WWRogosinski_1958a} (see also
\citep[Th.\ 7.37]{AShapiro_DDentcheva_ARuszczynski_2014a}),
the maximizing measure for \eqref{eq:MLE} is a convex
combination of at most $2M+1$ Dirac measures.  Consequently,
we can reformulate \eqref{eq:MLE} as the optimization problem
\begin{align*}
  &\max_{\substack{\theta_1,\ldots,\theta_{2M+1}\in\Theta \\ \alpha_1,\ldots,\alpha_{2M+1}\in\real}}
     \quad\sum_{m=1}^{2M+1} \alpha_m M_N(\theta_m) \\
  &\quad\operatorname{subject\;to}\quad
     \mathfrak{m}_k-\varepsilon_k\le\sum_{m=1}^{2M+1} \alpha_m\psi_k(\theta_m)
      \le\mathfrak{m}_k+\varepsilon_k, \;\;k=1,\ldots,2M+1 \\
  &  \;\qquad\qquad\qquad\sum_{m=1}^{2M+1}\alpha_m = 1,
       \;\;\alpha_k \ge 0, \;\;k=1,\ldots,2M+1.
\end{align*}
See \citep[Th.~6.66]{AShapiro_DDentcheva_ARuszczynski_2014a} for details
in the context of distributionally robust stochastic optimization.
Given optimal $\alpha_m^\star$ and $\theta_m^\star$ (if they exist),
the CGVI measure then
\[
  P^\star = \sum_{i=1}^{2M+1} \alpha_m^\star \delta_{\theta_m^\star}.
\]

\section{Conclusions}
Extended M-estimation provides a generalization of traditional statistical
estimation procedures including maximum likelihood, regression, Bayesian
inference, and the Gibbs posterior.
Much like Bayesian inference and the Gibbs posterior, CGVI
permits the use of subjective and data-driven information to produce a
distribution of likely values for the unknown parameters, which can then be
used to perform further analyses.  Additionally, these distributions often have
semi-analytical representations.  In particular, they require the solution of a
small convex optimization problem in the case of the $\phi$-divergence. Many natural
extensions of this work exist, including its application to estimation using
functional data (e.g., time-dependent signals), its application to training
machine learning models such as GANs, the development of efficient sampling
methods and the development of rigorous experimental design techniques that use
CGVI.

\acks{
  Sandia National Laboratories is a multimission laboratory
  managed and operated by National Technology and Engineering
  Solutions of Sandia, LLC., a wholly owned subsidiary of
  Honeywell International, Inc., for the U.S.\ Department of
  Energy's National Nuclear Security Administration under
  contract DE-NA0003525.
  This paper describes objective technical results and analysis. Any
  subjective views or opinions that might be expressed in the paper
  do not necessarily represent the views of the U.S.\ Department of
  Energy or the United States Government.
}

\appendix

\section{Auxiliary Results for $\phi$-Divergence CGVI}\label{s:ap_Orlicz}
In this appendix, we study the topological properties of $\fD_\varepsilon$ and
$v_{\fD_\varepsilon}$ defined using a $\phi$-divergence.

\subsection{Topological Properties of $\fD_\varepsilon$}\label{ss:ap_D}

First, we show that $\fD_\varepsilon\subset L_\psi$. To see this, we have the
inequality
\begin{equation}\label{eq:ap-order}
  \psi(x) = \sup_{t\ge 0}\{tx - \phi^*(t)\}
          \le \sup_{t\in\real}\{tx - \phi^*(t)\}
            = \phi(x),
\end{equation}
which implies $\bbe[\psi(p)] \le \bbe[\phi(p)] \le \varepsilon$ and
$\fD_\varepsilon\subset L_\psi$.
As a result, we choose $\dset=L_\psi$ and $\mset=L_{\psi^\#}$.
We note that both $\dset$ and $\mset$ are decomposable.  In particular, 
let $X\in L_{\psi}$, $A\in\cB$ and $Y$ a $\cB$-measurable and bounded function,
then there exists $a>0$ and $b>0$ such that $\bbe[\psi(|X/a|)] < \infty$ and
$\bbe[\psi(|Y/b|)] < \infty$. Now, let $c=\max\{a,b\}$ and define
\[
  Z \eqdef \left\{\begin{array}{ll}
    X & \text{in $\Theta\setminus A$} \\
    Y & \text{in $A$}
  \end{array}\right. .
\]
Then, $\bbe[\psi(|Z/c|)] < \infty$ and $Z\in L_{\psi}$ (cf.\
\citep[pg.~184]{RTRockafellar_1971a} for additional discussion).
The decomposability of $\dset$, combined with
\citep[Cor.~3D]{RTRockafellar_1976a}, ensures that
$\Phi(\cdot)=\bbe[\phi(\cdot)]$ is $\sigma(\dset,\mset)$-lower semicontinuous
and hence $\fD_\varepsilon$ is $\sigma(\dset,\mset)$-closed.
As a final result, we prove that $\fD_\varepsilon$ is bounded in $\dset$.
To this end, \citep[Th.~2.2.9]{GAEdgar_LSucheston_1992a} provides
\[
  \|p\|_{L_\psi} \le \|p\|_{L_{\psi^\#}}^\#
    = \sup\left\{|\bbe[p X]|\;\vert\;\bbe[\psi^\#(|X|)]\le 1\right\}
  \quad\forall\,p\in L_\psi,
\]
and applying Young's inequality to the objective function in the definition
of the norm $\|\cdot\|_{L^{\psi^\#}}^{\#}$ yields
\[
  |\bbe[p X]|\le\bbe[|p||X|]\le\bbe[\psi(|p|)]+\bbe[\psi^\#(|X|)]
  \le \bbe[|\psi(|p|)] + 1
  \quad\forall\,p\in L_\psi.
\]
Therefore, $\|p\|_{L_\psi}\le \varepsilon+1$ for all $p\in\fD_\varepsilon$
and $\fD_\varepsilon$ is bounded.

\subsection{Properties of $\cR^{\varepsilon}$}\label{ss:ap_sig}

To prove that $\cR^\varepsilon$ is finite valued on $\mset$, we first note
that $\cV(X) = \bbe[v(X)]$ and since $v$ is increasing we have that
$\cV(X) \le \bbe[v(|X|)]$. Now, if $\|X\|_{\mset}\le 1$, then $\cV(X)\le 1$.
On the other hand, if $\|X\|_{\mset}>1$, then $|X|/\|X\|_{\mset} < |X|$ a.s.\
and 

The goal of this section is to prove that $\cR^{\varepsilon}$ is
finite valued.  To do this, we first show that $\fD_\varepsilon$ is bounded in
$L_\psi$.
To this end, \citep[Th.~2.2.9]{GAEdgar_LSucheston_1992a} provides
\[
  \|p\|_{L_\psi} \le \|p\|_{L_{\psi^\#}}^\#
    = \sup\left\{|\bbe[p X]|\;\vert\;\bbe[\psi^\#(|X|)]\le 1\right\}
  \quad\forall\,p\in L_\psi,
\]
and applying Young's inequality to the objective function in the definition
of the norm $\|\cdot\|_{L^{\psi^\#}}^{\#}$ yields
\[
  |\bbe[p X]|\le\bbe[|p||X|]\le\bbe[\psi(|p|)]+\bbe[\psi^\#(|X|)]
  \le \bbe[|\psi(|p|)] + 1
  \quad\forall\,p\in L_\psi.
\]
Therefore, $\|p\|_{L_\psi}\le \varepsilon+1$ for all $p\in\fD_\varepsilon$
and $\fD_\varepsilon$ is bounded.  Using this, we can prove that
$\cR^{\varepsilon}$ is finite valued.  Let $X\in L_{\psi^\#}$, then by
\citep[Prop.~2.2.7]{GAEdgar_LSucheston_1992a}, we can bound the objective
function in \eqref{eq:cgvi-phi} by
\[
  \bbe[p X] \le \bbe[p|X|] \le 2\|p\|_{L_\psi}\|X\|_{L_{\psi^\#}}
                           \le 2(\varepsilon+1)\|X\|_{L_{\psi^\#}}
\]
for all $p\in\fD_\varepsilon$.
Therefore, $\cR^{\varepsilon}$ is finite valued on $L_{\psi^\#}$ and
Young's inequality ensures that
\[
  \bbe[p X] - \Phi(p)
  \le \bbe[\psi(p)] + \bbe[\psi^\#(|X|)] - \bbe[\phi(p)]
  \le \bbe[\psi^\#(|X|)]
\]
for all $p\in\dset$ with $p\ge 0$ a.s.  In particular, $\Phi^*$ is finite
and hence continuous and subdifferentiable.  Theorem~\ref{th:gen-exist} then
provides conditions for existence of solutions.  In particular, a CGVI density
exists in $\fD_\varepsilon$ if
\[
  \emptyset\neq \partial v_\varepsilon(M_N)\cap\fD_\varepsilon
  \subseteq L_{\psi}.
\]

\section{Asymptotic Consistency of SAA Estimators}\label{ap:consistency}
In this appendix, we prove various properties of the integrand in
\eqref{eq:cgvi} when $\Phi^*$ is an expectation.  These properties ensure the asymptotic consistency
of the estimators for $\bar{\lambda}$ and $\bar{\mu}$ in the large sample
limit (i.e., $S\to\infty$).
\begin{theorem}\label{thm:asymp}
Let the assumptions of Corollary~\ref{th:exist-phi} holds.  Then, we have the
following properties:
\begin{enumerate}
\item The integrand $F : \real \times[0,\infty)\times\Theta\to[-\infty,\infty]$
      given by
      \[
        F(\mu,\lambda,\theta) \eqdef  \mu + \lambda\varepsilon + (\lambda\phi)^*(M_N(\theta) - \mu)
      \]
      is random lower semicontinuous (also called a normal integrand
      \citep{RTRockafellar_RJBWets_1998a});
\item There exists $\Theta_0\in\mathfrak{B}$ with $\Pi(\Theta_0)=1$ such that
      $F(\cdot,\cdot,\theta):\real \times [0,\infty) \to [-\infty,\infty]$ is
      convex for all $\theta \in \Theta_0$;
\item The integral function $f:\real\times[0,\infty)\to[-\infty,\infty]$ given
      by
      \[
        f(\mu,\lambda) \eqdef \mathbb E[F(\mu,\lambda)]
      \]
      is lower semicontinuous and there exists
      $(\overline{\mu},\overline{\lambda})\in\real\times[0,\infty)$ and a
      neighborhood $U \subset \real\times[0,\infty)$ containing
      $(\overline{\mu},\overline{\lambda})$ on which $f$ is bounded from above;
\item The set of minimizers $X^{\star} \subset \real \times [0,\infty)$ of $f$
      is nonempty, closed, convex, and bounded;
\item The pointwise LLN holds for every
      $(\mu,\lambda) \in \real \times [0,\infty)$, i.e.
      \[
        \widehat{f}_{S}(\mu,\lambda)  \eqdef \mu + \lambda\varepsilon
         + \frac{1}{S}\sum_{s=1}^S(\lambda\phi)^*(M_N^s - \mu) 
         \to f(\mu,\lambda)
      \] 
      a.s.\ for all $(\mu,\lambda) \in \real \times [0,\infty)$.
\end{enumerate}
\end{theorem}
\begin{proof}
To prove Property 1, we first consider
$\widehat{F}:\real\times[0,\infty)\times\real\to[-\infty,\infty]$ where
\[
  \widehat{F}(\mu,\lambda,t) \eqdef  \mu + \lambda\varepsilon + (\lambda\phi)^*(t - \mu).
\]
Note that $F(\mu,\lambda,\theta) = \widehat{F}(\mu,\lambda,M_N(\theta))$ and
$\widehat{F}(\mu,\lambda,t)$ is otherwise independent of $\theta \in \Theta$.
According to \citep[Ex.~14.30~\&~14.31]{RTRockafellar_RJBWets_1998a}, we only
need to prove that $\widehat{F}$ is lower-semicontinuous in order for it to
qualify as a random lower-semicontinuous function. To this end, let
$(\mu_k,\lambda_k,t_k) \in \real\times[0,\infty)\times\real$ such
that $(\mu_k,\lambda_k,t_k) \to
(\overline{\mu},\overline{\lambda},\overline{t})$ as $k \to \infty$.
By definition of the Fenchel conjugate, we have
\[
\widehat{F}(\mu_k,\lambda_k,t_k) 
\ge
\mu_k + \varepsilon \lambda_k + (t_k - \mu_k) x - \lambda_k \phi(x),\quad \forall x \in \mathbb R.
\]
It follows that for any fixed $x \in \real$ the lower bound is either
finite or $-\infty$.  Since $\phi$ is proper, there is at least one point
$x\in\real$ for which $\phi(x) < \infty$. Taking the limit inferior of both
sides yields
\[
\liminf_{k \to +\infty} \widehat{F}(\mu_k,\lambda_k,t_k)  \ge 
\overline{\mu} + \varepsilon \overline{\lambda} + (\overline{t} - \overline{\mu}) x - \overline{\lambda} \phi(x).
\]
Taking the supremum over $x\in\real$ on the right hand side above yields 
\[
\liminf_{k \to +\infty} \widehat{F}(\mu_k,\lambda_k,t_k) \ge 
\overline{\mu} + \varepsilon \overline{\lambda}  + ( \overline{\lambda} \phi)^*(\overline{t} - \overline{\mu}) = \widehat{F}(\overline{\mu},\overline{\lambda},\overline{t}),
\]
i.e., $\widehat{F}$ is lower semicontinuous. Now, since $M_N$ is
$\mathfrak{B}$-measurable it follows from
\citep[Prop.~14.45c]{RTRockafellar_RJBWets_1998a}
that $F$ is random lower semicontinuous.

In Property 2, the convexity of $F(\cdot,\cdot,\theta)$ follows from
\citep[Ex.~3.49]{RTRockafellar_RJBWets_1998a}.
%

To prove Property~3, the feasible set in \eqref{eq:cgvi} is nonempty,
closed, and convex with a nonempty interior. In addition, for any
$(\mu,\lambda) \in \real\times[0,\infty)$ there exists $x\in\real$ satisfying 
\[
F(\mu,\lambda,\theta) \ge \mu + \varepsilon \lambda + (M_N(\theta) - \mu) x - \lambda \phi(x).
\]
In fact, $x$ can be chosen independently of $(\mu,\lambda)$ in the effective
domain of $\phi$. Therefore,
$F(\mu,\lambda,\cdot) : \Theta \to [-\infty,\infty]$ is bounded from below by
an integrable function. Combining this with the lower-semicontinuity of
$F(\cdot,\cdot,\theta) : \real \times [0,\infty) \to [-\infty,\infty]$ and
using Fatou's lemma, we deduce the lower semicontinuity of the integral $f$
on $\real\times[0,\infty)$.

To prove Property~4, let $(\overline{\mu},\overline{\lambda}) = (0,1)$, which
is clearly feasible. In addition, for $\Pi$-almost all $\theta \in \Theta$ and
any pair $(\mu,\lambda)$ in a small neighborhood around $(0,1)$, we have
\[
F({\mu},{\lambda},\theta) = \mu + \varepsilon \lambda + \lambda \phi^*((M_N(\theta)-\mu)/\lambda)
\]
By assumption, $M_N \in L_{\psi^{\#}}$ and since
$(\Theta,\mathcal{B},\Pi)$ is a probability space, the constant function
$\mu/\lambda \in  L_{\psi^{\#}}$. Hence, $(M_N(\cdot)-\mu)/\lambda \in L_{\psi^{\#}}$.
Using the fact that $\psi \le \phi$, we have $\psi^{\#} \ge \phi^*$ (as $\phi$
is only finite on $[0,\infty)$). This provides the upper bound
\[
F({\mu},{\lambda},\theta) \le \mu + \varepsilon \lambda + \lambda \psi^{\#}((M_N(\theta)-\mu)/\lambda)
\]
and therefore,
\[
f(\mu,\lambda) \le \mu + \varepsilon \lambda \mathbb E[ \psi^{\#}((M_N(\theta)-\mu)/\lambda)] < \infty.
\]

Finally, for Property~5, we first note that \eqref{eq:cgvi} holds by
Theorem~\ref{th:exist-phi} (cf.\ Appendix~\ref{s:ap_Orlicz}).  Then by strong
duality, $f$ admits at least one minimizer
$(\bar{\mu},\bar{\lambda})\in \real \times [0,\infty)$. Let $\overline{f} \eqdef
f(\bar{\mu},\bar{\lambda})$. Since $f$ is lower semicontinuous and convex on
$\real \times [0,\infty)$, the sublevel set 
\[
\left\{(\mu,\lambda) \in \mathbb R \times \mathbb R_+ \left|\; f(\mu,\lambda) \le \overline{f}\right.\right\}
\]
is closed and convex. Next, let $\left\{(\mu_k,\lambda_k)\right\}$ be a
sequence of minimizers.  Then for $\Pi$-almost all $\theta$, we have
\[
F(\mu_k,\lambda_k,\theta) \ge \mu_k + \varepsilon \lambda_k + (M_N(\theta) - \mu_k) x - \lambda_k \phi(x) \quad\forall\, x \in \real.
\]
Setting $x = 1$ yields
\[
F(\mu_k,\lambda_k,\theta) \ge  \varepsilon \lambda_k + M_N(\theta).
\]
Taking the expectation and using the fact that $\lambda_k \ge 0$, we have 
\[
0 \le \lambda_k \le \overline{f} - \mathbb E[M_N]
\]
It follows that $\left\{\lambda_k\right\}$ is bounded. Due to the assumptions
on $\phi$, we can readily argue for the existence of constants
$x_1,x_2\in\real$ with $0 < x_1 < 1< x_2$ such that $\phi(x_i)$ is finite for
$i=1,2$. Therefore, there exists a constant $c > 0$ such that
\[
(1 - x_i)\mu_k \le \overline{f} - \varepsilon \lambda_k - x_i\mathbb E[M_N] + \lambda_k \phi(x_i) < c
\]
for all $k$. Hence, $\left\{\mu_k\right\}$ is bounded as well.
Since $F$ is random lower semicontinuous and the samples are iid, then strong
the LLN holds for each fixed pair $(\mu,\lambda) \in \real \times [0,\infty)$,
completing the proof.
\end{proof}

\vskip 0.2in
\bibliography{references}

\end{document}